\documentclass[10pt,a4paper,reqno,english]{amsart}
\title[]{On large potential perturbations of the Schr\"{o}dinger, wave and Klein--Gordon equations}
\def\PpP{-}

\usepackage{%
    amssymb%
    ,babel%
    ,csquotes%
    ,graphicx%
    ,mathrsfs%
    ,eucal%
    ,xcolor%
    ,enumitem%
    ,geometry%
    ,esint}
\usepackage[files,\PpP]{pagesel}%
\usepackage[%
    pdfauthor={Piero D'Ancona}%
    ,colorlinks=true%
    ,linkcolor=magenta%
    ,citecolor=cyan%
    ,urlcolor=blue%
    ,hyperfootnotes=false%
]{hyperref}

\DeclareMathOperator{\spt}{spt}

\newcommand{\rie}{\mathsf{R}}
\newcommand{\bra}[1]{\langle #1 \rangle}
\newcommand{\one}[1]{\mathbf{1}_{#1}}

\usepackage{stmaryrd}
  {\baselineskip0pt\normalfont\footnotesize%
  \abovedisplayskip=1pt\abovedisplayshortskip=1pt%
  \belowdisplayskip=1pt\belowdisplayshortskip=1pt%
  \vskip2pt\par\noindent$\llbracket$\ignorespaces}%
  {\ignorespaces$\rrbracket$\vskip2pt}
\numberwithin{equation}{section}
\newtheorem{theorem}{Theorem}[section]
\newtheorem{corollary}[theorem]{Corollary}
\newtheorem{lemma}[theorem]{Lemma}

\theoremstyle{remark}
\newtheorem{remark}[theorem]{Remark}

\theoremstyle{definition}
\newtheorem{definition}[theorem]{Definition}

\date{\today}

\author[P.~D'Ancona]{Piero D'Ancona}
\address{Piero D'Ancona: 
Dipartimento di Matematica\\
Sapienza Universit\`{a} di Roma\\
Piazzale A.~Moro 2\\
00185 Roma\\
Italy}
\email{dancona@mat.uniroma1.it}

\thanks{%
}
\subjclass[2010]{%
35Q41
, 35L05
}
\keywords{%
Schr\"{o}dinger equation%
; Strichartz estimates%
; Dispersive equations%
; Resolvent estimates%
; Local energy decay%
}

\begin{document}

\begin{abstract}
  We prove a sharp resolvent estimate in scale invariant norms of Amgon--H\"{o}rmander type for a magnetic Schr\"{o}dinger operator on $\mathbb{R}^{n}$, $n\ge3$
  \begin{equation*}
    L=-(\partial+iA)^{2}+V
  \end{equation*}
  with large potentials $A,V$ of almost critical decay and regularity.

  The estimate is applied to prove sharp smoothing and Strichartz estimates for the Schr\"{o}dinger, wave and Klein--Gordon flows associated to $L$.
\end{abstract}

\maketitle



\section{Introduction}\label{sec:intr}

We consider a selfadjoint Schr\"{o}dinger operator 
in $L^{2}(\mathbb{R}^{n})$, $n\ge3$, of the form
\begin{equation}\label{eq:firstL}
  L=-(\partial+iA)^{2}+V
\end{equation}
where $A=(A_{1},\dots,A_{n}):\mathbb{R}^{n}\to \mathbb{R}^{n}$
is the magnetic potential and
$V:\mathbb{R}^{n}\to \mathbb{R}$ the electric potential.
In order to allow a unified treatment of the dispersive
equations corresponding to $L$, we shall
always assume $L\ge0$, although this assumption
can be relaxed. We are interested in 
the dispersive properties for solutions of the equations
\begin{equation}\label{eq:dispeq}
  i \partial_{t}u+Lu=0,
  \qquad
  \partial^{2}_{t}u+Lu=0,
  \qquad
  \partial^{2}_{t}u+(L+1)u=0,
\end{equation}
associated to the operator $L$. 

The critical behaviour for dispersion appears to be
$|A|\lesssim |x|^{-1}$, $|V|\lesssim|x|^{-2}$, and one of
our goals is to get as close as possible to this kind of 
singularity. All the results of the paper are valid under
the following assumption
(note however that weaker conditions 
are required in the course of the paper):

\medskip
\textbf{Assumption (L)}. 
Let $n\ge3$.
The operator $L$ in \eqref{eq:firstL}
is selfadjoint in $L^{2}(\mathbb{R}^{n})$ with domain
$H^{2}(\mathbb{R}^{n})$, non negative, 0 is not
a resonance for $L$, and writing
$w(x)=\bra{\log|x|}^{\mu}\bra{x}^{\delta}$
for some $\delta>0$, $\mu>1$,
\begin{equation}\label{eq:VAass}
  w(x)|x|^{2}
  (V-i \partial \cdot A)
  \in L^{\infty},
  \quad
  w(x)|x|\widehat{B}\in L^{\infty},
  \quad
  w(x)|x| A\in L^{\infty}\cap\dot H^{1/2}_{2n}
\end{equation}
where 
$\widehat{B}_{k}:=\sum_{j=1}^{n}
\frac{x_{j}}{|x|}(\partial_{j}A_{k}-\partial_{k}A_{j})$
is the tangential component of the magnetic field.
\medskip

It is well known that a resonance at 0 is an obstruction
to dispersion. 
The precise notion required here is the following:

\begin{definition}[Resonance]\label{def:reson}
  We say that $0$ is a \emph{resonance} for the operator
  $L$ if there exists a nonzero
  $v\in H^{2}_{loc}(\mathbb{R}^{n}\setminus0)\cap H^{1}_{loc}$
  solution of $Lv=0$ with the properties
  \begin{equation}\label{eq:resona}
    |x|^{\frac n2-2-\sigma}v\in L^{2}
    \quad\text{and}\quad 
    |x|^{\frac n2-1-\sigma}\partial v\in L^{2}
    \qquad \forall \sigma\in(0,\sigma_{0}).
  \end{equation}
  for some $\sigma_{0}>0$.
  The function $v$ is then called a \emph{resonant state}
  at 0 for $L$.
  (If $n\ge5$ this condition reduces to 0
  being an eigenvalue of $L$).
\end{definition}

A standard approach to the problem is based on a uniform estimate
for the resolvent operator of $L$.
This approach has a long tradition, starting from the
classical theories of Kato, Kato--Kuroda
and Agmon. 
The bulk of the paper 
(Sections \ref{sec:largefreq}--\ref{sec:the_full_reso_esti})
is devoted to prove the following estimate,
which is sharp even for $L=-\Delta$:

\begin{theorem}[Resolvent estimate]\label{the:1}
  Suppose Assumption (L) is verified. 
  Then the resolvent operator $R(z)=(L-z)^{-1}$
  satisfies the estimate
  \begin{equation}\label{eq:resestRz}
    \|R(z)f\|_{\dot X}
    +
    |z|^{\frac12}\|R(z)f\|_{\dot Y}
    +
    \|\partial R(z)f\|_{\dot Y}
    \lesssim
    \|f\|_{\dot Y^{*}}
  \end{equation}
  with a constant uniform in $z$ in the complex strip 
  $|\Im z|\le1$. In particular, the boundary values
  of $R(\lambda\pm i \epsilon)$ as $\epsilon \downarrow0$
  are well defined bounded operators from 
  $\dot Y^{*}$ to $\dot X$ (or to $\dot Y$ 
  provided $\lambda\neq0$).
\end{theorem}

See Theorem
\ref{the:fullresest} and Corollary \ref{cor:equivass} below.
Here the spaces $\dot X, \dot Y$ have norms
\begin{equation*}
  \textstyle
  \|v\|_{\dot{X}}^2:=\sup_{R>0}\frac1{R^2}
  \int_{|x|=R}|v|^2dS
  \qquad
  \|v\|_{\dot{Y}}^2:=\sup_{R>0}\frac1{R}
  \int_{|x|\leq R}|v|^2dx
\end{equation*}
while $\dot Y^{*}$ is the (pre)dual of $\dot Y$;
note that $\dot Y^{*}$ is an homogeneous
version of the Agmon--H\"{o}rmander space $B$
\cite{AgmonHormander76-a}. 
The last property in the statement is also called the 
\emph{limiting absorption principle} for $L$.
We think that an interesting contribution of the present paper is a
conceptually simple proof of \eqref{eq:resestRz},
based on a combination of the
multiplier method (for large frequencies) and Fredholm
theory (for small frequencies).

With \eqref{eq:resestRz} at our disposal, the
classical Kato's theory of smoothing operators gives 
with little effort several smothing estimates
(also known as \emph{local energy deca}y) for the
Schr\"{o}dinger flow $e^{itL}$.
Kato's theory was extended in 
\cite{DAncona15-a} to include the wave and Klein--Gordon
equations. By combining these techniques, we obtain
the following scaling invariant estimates:

\begin{theorem}[Smoothing estimates]\label{the:2}
  Under Assumption (L), we have:
  \begin{equation*}
    \||x|^{-1/2}e^{itL}f\|_{\dot YL^{2}_{t}}
    +
    \||D|^{1/2}e^{itL}f\|_{\dot YL^{2}_{t}}
    \le
    C\|f\|_{L^{2}},
  \end{equation*}
  \begin{equation*}
    \||x|^{-1/2}e^{it\sqrt{L}}f\|_{\dot YL^{2}_{t}}
    +
    \||D|^{1/2}e^{it \sqrt{L}}f\|_{\dot YL^{2}_{t}}
    \le
    C\|f\|_{\dot H^{1/2}},
  \end{equation*}
  \begin{equation*}
    \||x|^{-1/2}e^{it\sqrt{L+1}}f\|_{\dot YL^{2}_{t}}
    +
    \||D|^{1/2}e^{it \sqrt{L+1}}f\|_{\dot YL^{2}_{t}}
    \le
    C\|f\|_{H^{1/2}}.
  \end{equation*}
\end{theorem}

In the $\dot YL^{2}_{t}$ norm the order of integration
is reversed, but one can easily write these estimates
in a more standard (and actually equivalent) 
form in terms of $L^{2}$ 
weighted norms. Indeed, if $\rho$ is any function such that
$\sum_{j\in\mathbb{Z}}\|\rho\|_{L^{\infty}(|x|\sim 2^{j})}^{2}
  <\infty$, 
we have $\|\rho |x|^{-1/2}v\|_{L^{2}}\lesssim\|v\|_{\dot Y}$,
hence the smoothing estimates for Schr\"{o}dinger can be written
\begin{equation*}
  \|\rho|x|^{-1}e^{itL}f\|_{L^{2}_{t}L^{2}}
  +
  \|\rho|x|^{-\frac12}|D|^{\frac12}e^{itL}f\|_{L^{2}_{t}L^{2}}
  \lesssim
  \|f\|_{L^{2}}
\end{equation*}
and similarly for the wave and Klein--Gordon equation.
A typical example of such a weight is
$\rho=\bra{\log|x|}^{-\nu}$ for $\nu>1/2$.

These smoothing estimates, 
together with the corresponding inhomogeneous 
ones, are proved in Section \ref{sec:smoo_esti} and in particular
Corollary \ref{cor:smoosch}, \ref{cor:smooWE} 
and \ref{cor:smooswap}.
Note that if we are in the \emph{Coulomb gauge}
$\partial \cdot A=0$, 
the last condition in \eqref{eq:VAass}
is not necessary both for the smoothing estimates and
the uniform resolvent estimate \eqref{eq:resestRz}.

As a final application, in Section \ref{sec:stri_esti}
we prove the full set of Strichartz estimates for
the three dispersive equations \eqref{eq:dispeq}.
We recall the basic facts for the unperturbed case
in dimension $n\ge3$:
\begin{itemize}
  \item 
  A couple $(p,q)$ is
  \emph{Schr\"{o}dinger admissible} if
  \begin{equation*}
    \textstyle
    p\in[2,\infty],
    \qquad
    q\in[2,\frac{2n}{n-2}],
    \qquad
    \frac2p+\frac nq=\frac n2,
  \end{equation*}
  and \emph{wave admissible} if
  \begin{equation*}
    \textstyle
    p\in[2,\infty],
    \qquad
    q\in[2,\frac{2(n-1)}{n-3}],
    \qquad
    \frac2p+\frac {n-1}q=\frac {n-1}2,
    \qquad
    q\neq \infty.
  \end{equation*}
  \item 
  The \emph{homogeneous Strichartz estimates} are
  \begin{align*} 
    \|e^{-it \Delta}f\|_{L^{p}_{t}L^{q}}
    \lesssim
    \|f\|_{L^{2}},
    &
    \quad (p,q)\ \text{Schr\"{o}dinger admissible,}
  \\ 
    \||D|^{\frac1q-\frac1p}e^{-it|D|}f\|_{L^{p}_{t}L^{q}}
    \lesssim
    \|f\|_{\dot H^{\frac12}},
    &
    \quad (p,q)\ \text{wave admissible,}
  \\
    \|\bra{D}^{\frac1q-\frac1p}e^{it\bra{D}}f\|_{L^{p}_{t}L^{q}}
    \lesssim
    \|f\|_{H^{\frac12}},
    &
    \quad (p,q)\ \text{Schr\"{o}dinger or wave admissible.}
  \end{align*}
  Corresponding
  inhomogeneous versions of the estimates are also true.

  \item 
  The previous estimates can be refined using Lorentz norms.
  At the \emph{(Schr\"{o}dinger) endpoint}
  $(2,\frac{2n}{n-2})$ one gets
  \begin{equation}\label{eq:striendphom}
    \|e^{-it \Delta}f\|_{L^{2}_{t}L^{\frac{2n}{n-2},2}}
    \lesssim
    \|f\|_{L^{2}}
  \end{equation}
  and from this case all the other estimates can be 
  recovered, by interpolating with the conservation of  
  $L^{2}$ mass; 
  actually, by real interpolation one obtains estimates
  in the $L^{p}_{t}L^{q,2}$ norm for every admissible couple
  $(p,q)$.
  A similar situation occurs at the \emph{(wave) endpoint}
  $(2,\frac{2(n-1)}{n-3})$, in dimension $n\ge4$.
\end{itemize}

Then in Section \ref{sec:stri_esti} we prove:

\begin{theorem}[Strichartz estimates]\label{the:3}
  Suppose Assumption (L) is verified. Then we have
  the estimates
  \begin{equation*}
    \|e^{it L}f\|_{L^{2}_{t}L^{\frac{2n}{n-2},2}}
    \lesssim
    \|f\|_{L^{2}}
  \end{equation*}
  and hence the full set of $L^{p}_{t}L^{q}$ estimates,
  for all Schr\"{o}dinger admissible $(p,q)$;
  moreover, for all non endpoint, wave admissible couple $(p,q)$
  we have
  \begin{equation*}
    \||D|^{\frac1q-\frac1p}e^{-it \sqrt{L}}f\|_{L^{p}_{t}L^{q}}
    \lesssim
    \|f\|_{\dot H^{\frac12}}
  \end{equation*}
  and for all non endpoint, wave or Schr\"{o}dinger
  admissible couple $(p,q)$ we have
  \begin{equation*}
    \|\bra{D}^{\frac1q-\frac1p}e^{it\sqrt{L+1}}f\|_{L^{p}_{t}L^{q}}
    \lesssim
    \|f\|_{H^{\frac12}}.
  \end{equation*}
\end{theorem}

These estimates and their nonhomogeneous versions
are proved in Theorems \ref{the:endpointsch}, \ref{the:striWE} 
and \ref{the:striKG} in Section \ref{sec:stri_esti}.
Note that we prove the endpoint estimate for
the Schr\"{o}dinger equation, using a result
for the unperturbed Schr\"{o}dinger flow due to Ionescu and
Kenig \cite{IonescuKenig05-a} (which can be refined to
Lorentz spaces, as remarked in \cite{Mizutani16-b}).

\begin{remark}[]\label{rem:compare}
  We compare our results with
  \cite{ErdoganGoldbergSchlag09-a}, where for the first time
  smoothing and Strichartz estimates were obtained for
  Schr\"{o}dinger equations with large magnetic potentials,
  in any dimension $n\ge3$.
  The assumptions on the coefficients in 
  \cite{ErdoganGoldbergSchlag09-a} are
  \begin{equation*}
    |A|+\bra{x}|V|\lesssim \bra{x}^{-1-\epsilon},
    \qquad
    \bra{x}^{1+\epsilon'}A\in \dot H^{1/2}_{2n},
    \qquad
    A \ \text{is continuous.}
  \end{equation*}
  These conditions are largely overlapping with
  \eqref{eq:VAass}; we require a stronger condition
  on $\widehat{B}$, which is defined as a combination of first
  derivatives of $A$, but on the other hand we can consider
  potentials $A,V$ which are singular at the origin.
  Other improvements with respect to 
  \cite{ErdoganGoldbergSchlag09-a} are
  \begin{itemize}
    \item the endpoint Strichartz estimate for Schr\"{o}dinger;
    \item sharp scaling invariant resolvent and smoothing estimates;
    \item a unified treatment including wave and
    Klein--Gordon equations.
  \end{itemize}
  Last but not least, our proof is `elementary',
  indeed we use only multiplier 
  methods and Fredholm theory (and standard results from
  Calder\'{o}n--Zygmund theory).
  The only nonelementary result we need
  is Koch and Tataru's \cite{KochTataru06-a}
  to exclude embedded eigenvalues for $L$. 
  One can make the paper self--contained by
  assuming explicitly that no resonances exist
  in the spectrum of $L$.
  Note that under this additional assumption we can take 
  $\delta=0$ in \eqref{eq:VAass}, since the additional decay is
  used mainly to handle possible embedded resonances
  (see Lemma \ref{lem:eigv}).

  Note also that by a gauge transform it is possible to
  reduce to the case $\partial \cdot A=0$ i.e. to the
  \emph{Coulomb gauge}; see
  Remark \ref{rem:delta0},
  Corollary \ref{cor:equivass} and
  Corollary \ref{cor:strschgauge} for details.
  However,the quantity $\widehat{B}$ is gauge invariant
  and the assumption on $\widehat{B}$ 
  can not be removed by a change of gauge.
\end{remark}

We conclude with a short 
(and incomplete) summary of earlier results.
The case of purely electric potentials $A=0$ is well
understood; the list of papers on this topic is long
and here we mention only
\cite{JourneSofferSogge91-a},
\cite{Cuccagna00-a},
\cite{BurqPlanchonStalker04-a},
\cite{DAnconaPierfelice05-a},
and the series by Yajima
\cite{Yajima95-b},
\cite{Yajima95-a},
\cite{ArtbazarYajima00}
(see also \cite{DAnconaFanelli06-a}) concerning $L^{p}$ boundedness
of the scattering wave operator.
In particular, \cite{RodnianskiSchlag04-a}
introduced the strategy of proof used here,
based on Kato's theory (see also \cite{JourneSofferSogge91-a}).

The case of a \emph{small} magnetic potential $A$ was studied in
\cite{GeorgievStefanovTarulli05-a},
\cite{Stefanov07-a},
\cite{FanelliVega09-a}, and in
\cite{DAnconaFanelli08-a} where a comprehensive
study was done on
the main dispersive equations perturbed with a small magnetic 
and a large electric potential,
including massive and massless Dirac systems,
and
\cite{DAnconaFanelliVega10-a}.
See also \cite{Tataru08-a} where the case of fully variable 
coefficients is considered.

Smoothing and Strichartz estimates for the Schr\"{o}dinger
equation with a \emph{large} magnetic potential were proved in
\cite{ErdoganGoldbergSchlag09-a}--%
\cite{ErdoganGoldbergSchlag08-a}
(discussed above),
and for the wave equation in \cite{DAncona15-a},
where the resolvent estimates
of \cite{ErdoganGoldbergSchlag09-a} were used.

Standard references for Strichartz estimates,
at least for the Schr\"{o}dinger and wave equations, are
\cite{GinibreVelo85-a},
\cite{GinibreVelo95-a} and
\cite{KeelTao98-a}.
The situation for the Klein--Gordon flow is complicated by
the different scaling of $\bra{D}$
for small and large frequencies.
A complete analysis was made in
\cite{MachiharaNakanishiOzawa03-a}; 
a proof for Schr\"{o}dinger admissible $(p,q)$ 
can be found in \cite{DAnconaFanelli08-a}, while
wave admissible points can be deduced from the
precised dispersive estimate of \cite{Brenner85-a}.

\begin{remark}\label{rem:dirac}
  By similar techniques it is possible to prove smoothing
  and Strichartz estimates also for Dirac systems.
  This will be part of the joint work
  \cite{DAnconaOkamoto17-a}, concerning the cubic Dirac equation
  perturbed by a large magnetic potential.
\end{remark}

\section{The resolvent estimate for large frequencies}
\label{sec:largefreq}

We shall make constant use of the dyadic norms
\begin{equation}\label{eq:dyadicsp}
  \|v\|_{\ell^{p}L^{q}}
  :=
  \Bigl(
  \sum_{j\in \mathbb{Z}}
    \|v\|_{L^{q}(2^{j}\le|x|<2^{j+1})}^{p}
  \Bigr)^{1/p},
\end{equation}
with obvious modification when $p=\infty$.
More generally, we denote the
mixed radial--angular $L^{q}L^{r}$ norms on a
spherical ring $C=R_{1}\le|x|\le R_{2}$ with
\begin{equation*}
  \|v\|_{L^{q}_{|x|}L^{r}_{\omega}(C)}=
  \|v\|_{L^{q}L^{r}(C)}:=
  \textstyle
  (\int_{R_{1}}^{R_{2}}
    (\int_{|x|=\rho}|v|^{r}dS)^{q/r}d\rho)^{1/q}.
\end{equation*}
and we define for all $p,q,r\in[1,\infty]$
\begin{equation}\label{eq:lpLp}
  \|v\|_{\ell^{p}L^{q}L^{r}}
  :=
  \|\{\|v\|_{L^{q}L^{r}(2^{j}\le|x|<2^{j+1})}\}
  _{j\in \mathbb{Z}}\|_{\ell^{p}}.
\end{equation}
Clearly, when $q=r$ we have simply
$ \|v\|_{\ell^{p}L^{q}L^{q}}=\|v\|_{\ell^{p}L^{q}}$.
With these notations, the Banach norms appearing in
\eqref{eq:resestRz} can be equivalently defined as
\begin{equation*}
  \textstyle
  \|v\|_{\dot{X}}^2
  \simeq
  \||x|^{-1}v\|^{2}_{\ell^{\infty}L^{\infty}L^{2}},
  \qquad
  \textstyle
  \|v\|_{\dot{Y}}^2
  \simeq
  \||x|^{-1/2}v\|_{\ell^{\infty}L^{2}}^{2},
  \qquad
  \|v\|_{\dot Y^{*}}\simeq
  \||x|^{1/2}v\|_{\ell^{1}L^{2}}.
\end{equation*}

For large frequencies $|\Re z|\gg1$, we study the equation
\begin{equation}\label{eq:reseq}
  \Delta_{A}v+Wv+iZ \cdot \partial^{A} v+zv=f
\end{equation}
using a  direct approach based on the Morawetz multiplier 
method. Here
$A(x)=(A_{1}(x),\dots,A_{n}(x)):\mathbb{R}^{n}\to \mathbb{R}^{n}$,
$Z(x)=(Z_{1}(x),\dots,Z_{n}(x)):\mathbb{R}^{n}\to \mathbb{R}^{n}$,
$V:\mathbb{R}^{n}\to \mathbb{R}$,
and we use the notations
\begin{equation*}
  \textstyle
  \widehat{x}_{j}=\frac{x_{j}}{|x|},
  \quad
  \widehat{x}=\frac{x}{|x|},
  \quad
  \partial=(\partial_{1},\dots,\partial_{n}),
  \quad
  \partial^{A}=(\partial_{1}^{A},\dots,\partial_{n}^{A}).
\end{equation*}
\begin{equation*}
  \textstyle
  \Delta_{A}=\sum_{j=1}^{n}(\partial_{j}+iA_{j}(x))^{2},
  \qquad
  \partial_{j}=\frac{\partial}{\partial x_{j}},
  \qquad
  \partial_{j}^{A}=\partial_{j}+iA_{j}(x),
\end{equation*}
Recall that, using the convention of implicit summation over repeated
indices,
\begin{equation*}
  B_{jk}=\partial_{j}A_{k}-\partial_{k}A_{j},
  \qquad
  \widehat{B}_{j}=B_{jk}\widehat{x}_{k},
  \qquad
  \widehat{B}=(\widehat{B}_{1},\dots,\widehat{B}_{n}).
\end{equation*}
and we call the matrix $B$ the \emph{magnetic field} associated
to the potential $A(x)$, and $\widehat{B}$ the
\emph{tangential part} of the field.

We prove the following result:

\begin{theorem}[Resolvent estimate for large frequencies]\label{the:resestlarge}
  Let $n\ge3$. There exists a constant
  $\sigma_{0}$ depending only on $n$ such that
  the following holds. 
  
  Assume $v,f:\mathbb{R}^{n}\to \mathbb{C}$ satisfy
  \eqref{eq:reseq}, $W$ can be split as $W=W_{L}+W_{S}$,
  with
  $W_{L}(x),W_{S}(x),Z_{j}(x):\mathbb{R}^{n}\to \mathbb{R}$
  and 
  \begin{equation}\label{eq:asscoeff}
    \||x|^{3/2}W_{S}\|_{\ell^{1}L^{2}L^{\infty}}
    +\||x|Z\|_{\ell^{1}L^{\infty}}
    \le \sigma_{0},
    \quad
    |\Re z|
    \ge
    \sigma_{0}^{-1}
    \left[
    \||x|\widehat{B}\|_{\ell^{1}L^{\infty}}^{2}
    +
    \||x|W_{L}\|_{\ell^{1}L^{\infty}}
    \right]
    +2.
  \end{equation}
  Then the following estimate holds for all 
  $z$ as in \eqref{eq:asscoeff} with $|\Im z|\le1$
  \begin{equation}\label{eq:estlargela}
    \textstyle
    \|v\|_{\dot X}^{2}
    +
    (n-3)\|\frac{v}{|x|^{3/2}}\|_{L^{2}}^{2}
    +
    |z|\|v\|_{\dot Y}^{2}
    +
    \|\partial^{A}v\|_{\dot Y}^{2}
    \lesssim
    \|f\|_{\dot Y^{*}}^{2}
  \end{equation}
  with an implicit constant depending only on $n$.
\end{theorem}

\begin{remark}[]\label{rem:deAtode}
  Under a weak additional assumption on $A$,
  the norm $\|\partial^{A}v\|_{\dot Y}$ in
  \eqref{eq:estlargela} can be replaced by 
  $\|\partial v\|_{\dot Y}$, thanks to the following
\end{remark}

\begin{lemma}[]\label{lem:estdeAde}
  Assume $n\ge3$ and $A\in\ell^{\infty}L^{n}$.
  Then the following estimate holds
  \begin{equation}\label{eq:deAtode}
    \|\partial v\|_{\dot Y}
    \lesssim
    (1+\|A\|_{\ell^{\infty}L^{n}})
    \Bigl[\|\partial^{A}v\|_{\dot Y}+\|v\|_{\dot X}\Bigr].
  \end{equation}
\end{lemma}

\begin{proof}
  Let $C_{j}$ be the spherical shell $2^{j}\le|x|\le 2^{j+1}$
  and $\widetilde{C}_{j}=C_{j-1}\cup C_{j}\cup C_{j+1}$.
  Let $\phi$ be a nonnegative
  cutoff function equal to $1$ on $C_{j}$
  and vanishing outside $\widetilde{C}_{j}$, and let
  $\phi_{j}(x)=\phi(2^{-j}x)$. Then we can write
  \begin{equation*}
    \|\partial v\|_{L^{2}(C_{j})}\le
    \|\phi_{j}\partial v\|_{L^{2}}\le
    \|\phi_{j}\partial^{A} v\|_{L^{2}}
    +
    \|\phi_{j}Av\|_{L^{2}}
  \end{equation*}
  By H\"{o}lder's inequality and Sobolev embedding we have
  \begin{equation*}
    \|\phi_{j}Av\|_{L^{2}}\le
    \|A\|_{L^{n}(\widetilde{C}_{j})}
    \|\phi_{j}v\|_{L^{2}}
    \lesssim
    \|A\|_{\ell^{\infty}L^{n}}
    \|\phi_{j}v\|_{L^{\frac{2n}{n-2}}}
    \lesssim
    \|A\|_{\ell^{\infty}L^{n}}
    \|\partial(\phi_{j}|v|)\|_{L^{2}}.
  \end{equation*}
  We expand the last term as
  \begin{equation*}
    \|\partial(\phi_{j}v)\|_{L^{2}}
    \le
    \|(\partial\phi_{j})|v|\|_{L^{2}}
    +
    \|\phi_{j}(\partial|v|)\|_{L^{2}}.
  \end{equation*}
  We note that $|\partial \phi_{j}|\lesssim 2^{-j}$
  and we recall the pointwise diamagnetic inequality
  \begin{equation*}
    |\partial|v||\le|\partial^{A}v|
  \end{equation*}
  valid since $A\in L^{2}_{loc}$. Then we can write
  \begin{equation*}
    \|\partial(\phi_{j}v)\|_{L^{2}}
    \lesssim
    2^{-j}\|v\|_{L^{2}(\widetilde{C}_{j})}
    +
    \|\partial^{A}v\|_{L^{2}(\widetilde{C}_{j})}
    \lesssim
    2^{-j/2}\|v\|_{L^{\infty} L^{2}(\widetilde{C}_{j})}
    +
    \|\partial^{A}v\|_{L^{2}(\widetilde{C}_{j})}.
  \end{equation*}
  Summing up, we have proved
  \begin{equation*}
    \|\partial v\|_{L^{2}(C_{j})}
    \lesssim
    (1+\|A\|_{\ell^{\infty}L^{n}})
    \left[
    \|\partial^{A}v\|_{L^{2}(\widetilde{C}_{j})}
    +
    2^{-j/2}\|v\|_{L^{\infty} L^{2}(\widetilde{C}_{j})}.
    \right]
  \end{equation*}
  Multiplying both sides by $2^{-j/2}$ and taking the sup
  in $j\in \mathbb{Z}$ we get the claim.
\end{proof}

\subsection{Formal identities}
\label{sub:for_ide}
In the course of the proof we shall reserve the symbols
\begin{equation*}
  \lambda=\Re z,
  \qquad
  \epsilon=\Im z
\end{equation*}
for the components of the frequency $z=\lambda+i \epsilon$
in \eqref{eq:reseq}.

We recall two formal identities which are a special
case of the identities in
\cite{CacciafestaDAnconaLuca16-a}
(see also \cite{CacciafestaDAnconaLuca16-b}):
for any real valued weigths $\phi(x)$, $\psi(x)$, we have
(using implicit summation)
\begin{equation}\label{eq:id1}
\begin{split}
  \Re \partial_{j}Q_{j}
  =
  \Re(\Delta_{A}w+(\lambda+i \epsilon) w)
  \overline{[\Delta_{A},\psi]w}
  &
  \textstyle
  -\frac12 \Delta^{2}\psi|w|^{2}
  +2 \partial^{A}_{j}w\, (\partial_{j}\partial_{k}\psi) \,
  \overline{\partial^{A}_{k}w}
  \\
  &
  +2\Im( \overline{w}\,B_{jk}\,\partial^{A}_{j}w \,\partial_{k}\psi)
  +2 \epsilon
  \Im(w\,\partial_{j} \psi \, \overline{\partial^{A}_{j}w})
\end{split}
\end{equation}
and
\begin{equation}\label{eq:id2}
  \textstyle
  \partial_{j}P_{j}=
  \overline{w}\, \Delta_{A}w\,\phi
  +|\partial^{A}w|^{2}\, \phi
  -\frac12 \Delta \phi\,|w|^{2}
  +i\Im(\partial^{A}_{j}w\,\partial_{j}\phi\,\overline{w})
\end{equation}
where the quantities $Q=(Q_{1},\dots,Q_{n})$ and 
$P=(P_{1},\dots,P_{n})$ are defined by
\begin{equation*}
  \textstyle
  Q_{j}:=
  \partial_{j}^{A}w \, \overline{[\Delta_{A},\psi]w}
  -\frac12 \partial_{j}\Delta \psi\, |w|^{2}
  +\partial_{j}\psi\, 
  [\lambda|w|^{2}-|\partial^{A}w|^{2}]
\end{equation*}
\begin{equation*}
  \textstyle
  P_{j}:=\partial^{A}_{j}w\, \overline{w}\,\phi
  -\frac12 \partial_{j}\phi|w|^{2}.
\end{equation*}
Both formulas are easily checked by
expanding the terms in divergence form;
they are actually Morawetz type identities corresponding to
the two multipliers
\begin{equation*}
  \overline{[\Delta_{A},\psi]w}=
  (\Delta \psi)\overline{w}+2 \partial \psi \cdot 
  \overline{\partial^{A}w}
  \qquad\text{and}\qquad 
  \phi \overline{w}.
\end{equation*}

If we write equation \eqref{eq:reseq} in the form
\begin{equation}\label{eq:reducedg}
  \Delta_{A}v+(\lambda+i \epsilon)v=g,
  \quad\text{where}\quad 
  g:=f-Wv-iZ \cdot \partial^{A}v
\end{equation}
and we apply \eqref{eq:id1}, \eqref{eq:id2}, we obtain
\begin{equation}\label{eq:fundid}
  \textstyle
  \Re \partial_{j}
  \{(Q_{j}+P_{j})\}
  =
  I_{\nabla v}+I_{v}+I_{\epsilon}+I_{B}+I_{g}
\end{equation}
where
\begin{equation*}
  I_{\nabla v}=
  2 \partial^{A}_{j}v
  \, (\partial_{j}\partial_{k}\psi) \,
    \overline{\partial^{A}_{k}v}
  + \phi|\partial^{A}v|^{2},
  \qquad
  \textstyle
  I_{v}=
  -\frac12 \Delta(\Delta \psi+\phi)|v|^{2}
  -\lambda \phi|v|^{2}
\end{equation*}
\begin{equation*}
  I_{B}=
  2\Im( \overline{v}\,B_{jk}
  \,\partial^{A}_{j}v 
  \,\partial_{k}\psi),
  \qquad
  I_{\epsilon}=
  2 \epsilon\Im(v\,\partial_{j} 
  \psi \, \overline{\partial^{A}_{j}v}),
\end{equation*}
\begin{equation*}
  I_{g}=
  \Re(g\,\overline{[\Delta_{A},\psi]v}
  +g\,\overline{v}\,\phi)
\end{equation*}
In the following we shall integrate these formulas on 
$\mathbb{R}^{n}$ and use the fact that the boundary terms
vanish after integration. This procedure can be justified
in each case e.g. by approximating $v$ with smooth
compactly supported functions
and then extending the resulting estimates by density.
We omit the details which are standard.

\subsection{Preliminary estimates}\label{sub:pre_est}

Choosing $\phi=1$ in \eqref{eq:id2},
substituting \eqref{eq:reducedg} and taking the
imaginary part, we get
\begin{equation*}
  \epsilon|v|^{2}=
  \Im(g\overline{v})
  -\Im \partial_{j}\{\overline{v}
  \,\partial_{j}^{A}v\}
\end{equation*}
and after integration on $\mathbb{R}^{n}$ we obtain
\begin{equation}\label{eq:aux1}
  \textstyle
  \epsilon\|v\|_{L^{2}}^{2}=
  \Im\int g\overline{v}.
\end{equation}
Taking instead the real part of the same identity 
(also with $\phi=1$) we obtain
\begin{equation*}
  |\partial^{A}v|^{2}=
  \lambda|v|^{2}-\Re(g\overline{v})
  +\Re\partial_{j}\{\overline{v}
  \,\partial_{j}^{A}v\}
\end{equation*}
and after integration
\begin{equation}\label{eq:aux2}
  \textstyle
  \|\partial^{A}v\|_{L^{2}}^{2}
  =
  \lambda\|v\|_{L^{2}}^{2}
  -\Re\int g\overline{v}.
\end{equation}

In order to estimate the term $I_{\epsilon}$ 
in \eqref{eq:fundid} we use
\eqref{eq:aux1} and \eqref{eq:aux2} as follows:
\begin{equation*}
  \textstyle
  \int I_{\epsilon}
  \le
  2 |\epsilon|\|\partial \psi\|_{L^{\infty}}
  \|v\|_{L^{2}}\|\partial^{A}v\|_{L^{2}}
  \le
  C|\epsilon|^{1/2}(\int|g\overline{v}|)^{1/2}
  (|\lambda|\|v\|^{2}_{L^{2}}+\int|g\overline{v}|)
  ^{1/2}
\end{equation*}
with $C=2\|\partial \psi\|_{L^{\infty}}$,
then again by \eqref{eq:aux1}
\begin{equation*}
  \textstyle
  \le
  C(\int|g\overline{v}|)^{1/2}
  (|\lambda|\int|g\overline{v}|
  +|\epsilon|\int|g\overline{v}|)^{1/2}
\end{equation*}
and we arrive at the estimate
\begin{equation}\label{eq:estIep}
  \textstyle
  \int I_{\epsilon}
  \le
  2\|\partial \psi\|_{L^{\infty}}(|\lambda|+|\epsilon|)^{1/2}
  \|g\overline{v}\|_{L^{1}}.
\end{equation}

Another auxiliary estimate will cover 
the (easy) case of negative
$\lambda=-\lambda_{-}\le0$. Write 
the real part of identity \eqref{eq:id2} in the form
\begin{equation*}
  \textstyle
  \lambda_{-}|v|^{2}\phi+|\partial^{A}v|^{2}\phi
  -\frac12 \Delta \phi|v|^{2}
  =
  \sum_{\alpha}
  \partial_{j}\Re P_{j}
  -\Re(g_{\alpha}\overline{v}_{\alpha})\phi
\end{equation*}
and choose the radial weight
\begin{equation*}
  \textstyle
  \phi=\frac{1}{|x|\vee R}
  \quad\implies\quad
  \phi'=-\frac{1}{|x|^{2}}\one{|x|>R},
  \quad
  \phi''=-\frac{1}{R^{2}}\delta_{|x|=R}+
    \frac{2}{|x|^{3}}\one{|x|>R}.
\end{equation*}
Note that
\begin{equation*}
  \textstyle
  -\Delta \phi=
  \frac{1}{R^{2}}\delta_{|x|=R}
  +\frac{n-3}{|x|^{3}}\one{|x|>R}.
\end{equation*}
Integrating over $\mathbb{R}^{n}$ and taking the
supremum over $R>0$ we obtain the estimate for the case of negative
$\lambda=-\lambda_{-}\le0$
\begin{equation}\label{eq:aux4}
  \textstyle
  \lambda_{-}\|v\|_{\dot Y}^{2}
  +
  \|\partial^{A}v\|_{\dot Y}^{2}
  +
  \frac12\|v\|_{\dot X}^{2}
  + \frac{n-3}{2}\||x|^{-3/2}v\|_{L^{2}}^{2}
  \le
  \||x|^{-1}g\overline{v}\|_{L^{1}}.
\end{equation}

\subsection{The main terms}\label{sub:mai_est}

In the following we assume 
$|\epsilon|\le1$ and $\lambda\ge2$.
We choose in \eqref{eq:fundid}, for arbitrary $R>0$,
\begin{equation}\label{eq:ourpsi}
  \psi=
  \frac{1}{2R}|x|^{2}\one{|x|\le R}+|x|\one{|x|>R},
  \qquad
  \phi=-\frac{1}{R}\one{|x|\le R}.
\end{equation}
We have then
\begin{equation}\label{eq:Apsifi}
  \psi'=\frac{|x|}{|x|\vee R},
  \qquad
  \psi''=
  \frac1R\one{|x|\le R},
  \qquad
  \Delta \psi+\phi=\frac{n-1}{|x|\vee R},
\end{equation}
\begin{equation*}
  \textstyle
  \Delta(\Delta \psi+\phi)=
    -\frac{n-1}{R^{2}}
    \delta_{|x|=R}
    -\frac{(n-1)(n-3)}{|x|^{3}}
    \one{|x|>R}
\end{equation*}
This implies
\begin{equation}\label{eq:estIv}
  \textstyle
  3
  \sup_{R>0}\int I_{v}
  \ge
  \frac{n-1}{2}\|v\|_{\dot X}^{2}
  + (n-3)\||x|^{-3/2}v\|^{2}_{L^{2}}
  + \lambda\|v\|_{\dot Y}^{2}
\end{equation}

Next we can write, since $\psi$ is radial,
\begin{equation*}
  \textstyle
  2 \partial^{A}_{j}v\, (\partial_{j}\partial_{k}\psi) \,
    \overline{\partial^{A}_{k}v}
  =
  2\psi''
  \left|\widehat{x}\cdot \partial^{A}v\right| ^{2}
  +
  2\frac{\psi'}{|x|}
  \left[|\partial^{A}v|^{2}
    -\left|\widehat{x}\cdot \partial^{A}v\right| ^{2}\right]
  \ge
  \frac{2}{R}\one{|x|<R}|\partial^{A}v|^{2}.
\end{equation*}
This implies
\begin{equation}\label{eq:estInav}
  \textstyle
  \sup_{R>0}
  \int I_{\nabla v}
  \ge
  2\|\partial^{A}v\|_{\dot Y}^{2}.
\end{equation}

Further we have, since 
$B_{jk}\partial_{k}\psi
  =B_{jk}\widehat{x}_{k}\psi'=\widehat{B}_{j}\psi'$,
\begin{equation*}
  |I_{B}|\le \frac{2|x|}{|x| \vee R}|v||\partial^{A}v||\widehat{B}|
  \le 2|v||\partial^{A}v||\widehat{B}|
\end{equation*}
which implies
\begin{equation*}
  \textstyle
  \int |I_{B}|\le
  2
  \||x|\widehat{B}\|_{\ell^{1}L^{\infty}}
  \||x|^{-1/2}\partial^{A}v\|_{\ell^{\infty}L^{2}}
  \||x|^{-1/2}v\|_{\ell^{\infty}L^{2}}
  =
  2
  \||x|\widehat{B}\|_{\ell^{1}L^{\infty}}
  \|\partial^{A}v\|_{\dot Y}
  \|v\|_{\dot Y}
\end{equation*}
and by Cauchy--Schwartz, for any $\delta>0$,
\begin{equation}\label{eq:estIB}
  \textstyle
  \int |I_{B}|\le
  \delta\|\partial^{A}v\|_{\dot Y}^{2}
  +
  \delta^{-1}
  \||x|\widehat{B}\|_{\ell^{1}L^{\infty}}^{2}
  \|v\|_{\dot Y}^{2}.
\end{equation}

Finally, since $|\Delta \psi+\phi|\le(n-1)|x|^{-1}$ and
$|\partial \psi|\le1$, we have
\begin{equation}\label{eq:estIg}
  \textstyle
  \int|I_{g}|\le
  (n-1)\||x|^{-1}g\overline{v}\|_{L^{1}}
  +
  2\|g\overline{\partial^{A}v}\|_{L^{1}}.
\end{equation}

Summing up, by integrating identity \eqref{eq:fundid} over
$\mathbb{R}^{n}$ and using estimates
\eqref{eq:estIep}
\eqref{eq:estIv}, \eqref{eq:estInav}, \eqref{eq:estIB}
and \eqref{eq:estIg} we obtain
(recall that $|\partial \psi|\le1$; recall also
that $\lambda\ge2$ and $|\epsilon|\le1$ so that
$|\epsilon|+|\lambda|\lesssim \lambda$)
\begin{equation*}
\begin{split}
  \textstyle
  \|v\|_{\dot X}^{2}
  +
  &
  (n-3)\||x|^{-3/2}v\|_{L^{2}}^{2}
  + 
  \lambda\|v\|_{\dot Y}^{2}
  +
  \|\partial^{A}v\|_{\dot Y}^{2}
  \lesssim
  \\
  &
  \lesssim
  \delta\|\partial^{A}v\|_{\dot Y}^{2}
  +
  \delta^{-1}
  \||x|\widehat{B}\|_{\ell^{1}L^{\infty}}^{2}
  \|v\|_{\dot Y}^{2}
  +
  \lambda^{1/2}
    \|g\overline{v}\|_{L^{1}}
  +
  \||x|^{-1}g\overline{v}\|_{L^{1}}
  +
  \|g\overline{\partial^{A}v}\|_{L^{1}}
\end{split}
\end{equation*}
where $\delta>0$ is arbitrary
and the implicit constant depends only on $n$.
Note now that if $\delta$ is chosen small enough with 
respect to $n$ and we assume
\begin{equation}\label{eq:firstla}
  \lambda\ge c(n)
  \||x|\widehat{B}\|_{\ell^{1}L^{\infty}}^{2}
\end{equation}
for a suitably large $c(n)$, we can absorb two terms at the
right and we get the estimate
\begin{equation}\label{eq:firststep}
  \textstyle
  \|v\|_{\dot X}^{2}
  +
  \mu_{n}\|\frac{v}{|x|^{3/2}}\|_{L^{2}}^{2}
  +
  \lambda\|v\|_{\dot Y}^{2}
  +
  \|\partial^{A}v\|_{\dot Y}^{2}
  \le
  c_{0}\left(
  \|\frac{g\overline{v}}{|x|}\|_{L^{1}}
  +
  \|g\overline{\partial^{A}v}\|_{L^{1}}
  +  \lambda^{\frac12}
    \|g\overline{v}\|_{L^{1}}
    \right)
\end{equation}
where $c_{0}\ge1$ is a constant depending only on $n$.

\subsection{Conclusion} \label{sub:the_arg_for_lar_fre}

We now substitute in estimate \eqref{eq:firststep}
\begin{equation*}
  g=f-W(x)v- Z(x)\cdot\partial^{A}v
\end{equation*}
(see \eqref{eq:reducedg}).
Consider the terms at the right in 
\eqref{eq:firststep}, recalling that
\begin{equation*}
  W=W_{S}+W_{L}.
\end{equation*}
We denote by $\gamma,\Gamma$ the quantities
\begin{equation*}
  \gamma:=\||x|^{3/2}W_{S}\|_{\ell^{1}L^{2}L^{\infty}}
  +\||x|Z\|_{\ell^{1}L^{\infty}},
  \qquad
  \Gamma:=\||x|W_{L}\|_{\ell^{1}L^{\infty}}.
\end{equation*}
Then we have
\begin{equation*}
  \||x|^{-1}gv\|_{L^{1}}
  \le
  \||x|^{-1}W(x)v^{2}\|_{L^{1}}
  +
  \||x|^{-1}Z(x)\cdot \partial^{A}v\|_{L^{1}}
  +
  \||x|^{-1}f\overline{v}\|_{L^{1}}
\end{equation*}
and, for any $\delta>0$,
\begin{equation*}
\begin{split}
  \||x|^{-1}W(x)v^{2}\|_{L^{1}}
  \le&
  \||x|^{1/2}W_{L}\|_{\ell^{1}L^{2}L^{\infty}}
  \|v\|_{\dot X}
  \|v\|_{\dot Y}
  +
  \||x|W_{S}\|_{\ell^{1}L^{1}L^{\infty}}\|v\|_{\dot X}^{2}
  \\
  \le&
  (\delta+\gamma)\|v\|_{\dot X}^{2}+
  \delta^{-1}\Gamma^{2}\|v\|_{\dot Y}
\end{split}
\end{equation*}
\begin{equation*}
  \||x|^{-1}Z(x)\cdot \partial^{A}v\|_{L^{1}}
  \le
  \||x|Z\|_{\ell^{1}L^{\infty}}
  \|v\|_{\dot Y}\|\partial^{A}v\|_{\dot Y}
  \le
  \gamma^{2}\|\partial^{A}v\|_{\dot Y}^{2}
  +
  \gamma^{2}\|v\|_{\dot Y}^{2}
\end{equation*}
\begin{equation*}
  \||x|^{-1}f\overline{v}\|_{L^{1}}
  \le
  \|f\|_{\dot Y^{*}}\|v\|_{\dot X}
  \le \delta\|v\|_{\dot X}^{2}
  +\delta^{-1}\|f\|_{\dot Y^{*}}^{2}.
\end{equation*}
In a similar way we have
\begin{equation*}
  \|g\partial^{A} v\|_{L^{1}}
  \le
  \|W(x)v\partial^{A} v\|_{L^{1}}
  +
  \|Z(x) (\partial^{A}v)^{2}\|_{L^{1}}
  +
  \|f\overline{\partial^{A}v}\|_{L^{1}},
\end{equation*}
and
\begin{equation*}
\begin{split}
  \|W(x)v\partial^{A} v\|_{L^{1}}
  \le&
  \||x|W_{L}\|_{\ell^{1}L^{\infty}}
  \|v\|_{\dot Y}\|\partial^{A}v\|_{\dot Y}
  +
  \||x|^{3/2}W_{S}\|_{\ell^{1}L^{2}L^{\infty}}
  \|v\|_{\dot X}\|\partial^{A}v\|_{\dot Y}
  \\
  \le&
  2\delta\|\partial^{A}v\|_{\dot Y}^{2}
  +
  \delta^{-1}\Gamma^{2}\|v\|_{\dot Y}^{2}
  +
  \delta^{-1}\gamma^{2}\|v\|_{\dot X}^{2},
\end{split}
\end{equation*}
\begin{equation*}
  \|Z(x) (\partial^{A}v)^{2}\|_{L^{1}}
  \le
  \||x|Z\|_{\ell^{1}L^{\infty}}
  \|\partial^{A}v\|_{\dot Y}^{2}
  \le
  \gamma\|\partial^{A}v\|_{\dot Y}^{2},
\end{equation*}
\begin{equation*}
  \|f\overline{\partial^{A}v}\|_{L^{1}}
  \le
  \delta\|\partial^{A}v\|_{\dot Y}^{2}
  +
  \delta^{-1}\|f\|_{\dot Y^{*}}^{2}.
\end{equation*}
Finally we have
\begin{equation*}
  \lambda^{1/2}\|g \overline{v}\|_{L^{1}}
  \le
  \lambda^{1/2}\|W v^{2}\|_{L^{1}}
  +
  \lambda^{1/2}\|Z v\partial^{A}v\|_{L^{1}}
  +
  \lambda^{1/2}\|f \overline{v}\|_{L^{1}}
\end{equation*}
and
\begin{equation*}
\begin{split}
  \lambda^{1/2}\|W v^{2}\|_{L^{1}}
  \le&
  \lambda^{1/2}
  \||x|W_{L}\|_{\ell^{1}L^{\infty}}\|v\|_{\dot Y}^{2}
  +
  \lambda^{1/2}
  \||x|^{3/2}W_{S}\|_{\ell^{1}L^{2}L^{\infty}}
  \|v\|_{\dot X}\|v\|_{\dot Y}
  \\
  \le&
  (\lambda^{1/2}\Gamma+\lambda \gamma)\|v\|_{\dot Y}^{2}
  +
  \gamma\|v\|_{\dot X}^{2},
\end{split}
\end{equation*}
\begin{equation*}
  \lambda^{1/2}\|Z v\partial^{A}v\|_{L^{1}}
  \le
  \lambda^{1/2}
  \||x|Z\|_{\ell^{1}L^{\infty}}
  \|v\|_{\dot Y}\|\partial^{A}v\|_{\dot Y}
  \le
  \lambda \gamma\|v\|_{\dot Y}^{2}
  +
  \gamma \|\partial^{A}v\|_{\dot Y}^{2},
\end{equation*}
\begin{equation*}
  \lambda^{1/2}\|f \overline{v}\|_{L^{1}}
  \le
  \delta \lambda\|v\|_{\dot Y}^{2}
  +
  \delta^{-1}\|f\|_{\dot Y^{*}}^{2}.
\end{equation*}
Summing up, we get
\begin{equation*}
\begin{split}
  \||x|^{-1}gv\|_{L^{1}}
  +&
  \|g\partial^{A} v\|_{L^{1}}
  +
  \lambda^{1/2}\|g \overline{v}\|_{L^{1}}
  \le
  (\delta+2\gamma+\delta^{-1}\gamma^{2})\|v\|_{\dot X}^{2}
  \\
  +
  (2\delta^{-1}
  &
  \Gamma^{2}+
  \gamma^{2}+
    \lambda^{1/2}\Gamma+2\lambda \gamma+\delta \lambda)
    \|v\|_{\dot Y}^{2}
  +
  (\gamma^{2}+4\delta+2\gamma)\|\partial^{A}v\|_{\dot Y}^{2}
  +
  3\delta^{-1}\|f\|_{\dot Y^{*}}^{2}
\end{split}
\end{equation*}
Recalling that $c_{0}\ge1$ is the constant in
\eqref{eq:firststep}, depending only on $n$,
we require that
\begin{equation}\label{eq:constrconst}
  \textstyle
  \delta=\frac{1}{16c_{0}},
  \qquad
  \gamma\le\frac{1}{16c_{0}},
  \qquad
  |\lambda|\ge2^{8}c_{0}^{2}\Gamma^{2}+2
  +c(n)\||x|\widehat{B}\|_{\ell^{1}L^{\infty}}^{2}
\end{equation}
(note that this implies also \eqref{eq:firstla}
and $\lambda\ge2$)
and one checks that
\begin{equation*}
  \textstyle
  \delta+2 \gamma+\delta^{-1}\gamma^{2}\le \frac{1}{2c_{0}},
  \qquad
  \gamma^{2}+4\delta+2\gamma\le \frac{1}{2c_{0}}
\end{equation*}
and
\begin{equation*}
  \textstyle
  2\delta^{-1} \Gamma^{2}+ \gamma^{2}+
    \lambda^{1/2}\Gamma+2\lambda \gamma+\delta \lambda
  \le \frac{\lambda}{2c_{0}}.
\end{equation*}
Thus with the choices \eqref{eq:constrconst} we have
for positive $\lambda$
\begin{equation*}
  \textstyle
  \||x|^{-1}gv\|_{L^{1}}
  +
  \|g\partial^{A} v\|_{L^{1}}
  +
  \lambda^{1/2}\|g \overline{v}\|_{L^{1}}
  \le
  \frac{1}{2c_{0}}\|v\|_{\dot X}^{2}
  +
  \frac{\lambda}{2c_{0}} \|v\|_{\dot Y}^{2}
  +
  \frac{1}{2c_{0}}\|\partial^{A}v\|_{\dot Y}^{2}
  +
  3\delta^{-1}\|f\|_{\dot Y^{*}}^{2}
\end{equation*}
and plugging this into \eqref{eq:firststep}, and absorbing
the first three terms at the right from the left side of the 
inequality, we conclude that
\begin{equation}\label{eq:secondstep}
  \textstyle
  \|v\|_{\dot X}^{2}
  +
  \mu_{n}\|\frac{v}{|x|^{3/2}}\|_{L^{2}}^{2}
  +
  \lambda\|v\|_{\dot Y}^{2}
  +
  \|\partial^{A}v\|_{\dot Y}^{2}
  \le
  c_{1}\|f\|_{\dot Y^{*}}^{2}
\end{equation}
with $c_{1}$ a constant depending only on $n$.

Note that for \emph{negative} $\lambda$, starting from
estimate \eqref{eq:aux4} instead of \eqref{eq:firststep}
and applying the same argument, we obtain a similar 
estimate, provided $\lambda$ satisfies
\eqref{eq:constrconst}. Since out assumptions imply
$|\epsilon|\le|\lambda|$, we see that the proof of
Theorem \ref{the:resestlarge} is concluded.

\section{The resolvent estimate for small frequencies}
\label{sec:the_sma_fre_reg}

We now consider the remaining case of small requencies;
more precisely, we shall prove an estimate for all $z$
which is uniform for $z$ varying in any bounded region.
Define an operator $H$ as
\begin{equation}\label{eq:defL}
  Hv:=
  -\Delta v
  -
  W(x)v
  -
  iA \cdot\partial_{j}v
  -
  i\partial \cdot(A(x)v)
\end{equation}
with $W:\mathbb{R}^{n}\to \mathbb{R}$,
$A:\mathbb{R}^{n}\to \mathbb{R}^{n}$, and assume that $H$
is selfadjoint on $L^{2}(\mathbb{R}^{n};\mathbb{C}^{N})$.
In order to estimate the resolvent operator of $H$
\begin{equation*}
  R(z):=(H-z)^{-1}=
  (-\Delta-W-iA \cdot \partial-i \partial \cdot A-z)^{-1}
\end{equation*}
we use the (Lippmann--Schwinger) formula
\begin{equation}\label{eq:LSE}
  R(z)=R_{0}(z)(I-K(z))^{-1},
  \qquad
  K(z):=[W+iA \cdot \partial+i \partial \cdot A]R_{0}(z)
\end{equation}
expressing $R(z)$
in terms of the free resolvent
\begin{equation*}
  R_{0}(z)=(-\Delta-z)^{-1}.
\end{equation*}

We recall a few, more or less standard, facts on the
free resolvent $R_{0}(z)$. For 
$z\in \mathbb{C} \setminus[0,+\infty)$, $R_{0}(z)$
is a holomorphic map with values in the space of bounded 
operators $L^{2}\to H^{2}$ and satisfies an estimate
\begin{equation}\label{eq:freeest}
  \|R_{0}(z)f\|_{\dot X}
  +
  |z|^{\frac12}\|R_{0}(z)f\|_{\dot Y}
  +
  \|\partial R_{0}(z)f\|_{\dot Y}
  \lesssim
  \|f\|_{\dot Y^{*}}
\end{equation}
with an implicit constant independent of $z$
(sharp resolvent estimates can probably be traced
back to \cite{KenigPonceVega93-a}. A complete proof is 
given e.g. in
\cite{CacciafestaDAnconaLuca16-a}; actually
\eqref{eq:freeest} is a special case of the
computations in the previous Section for zero potentials, in which
case the proof given above works with no restriction on the
frequency).
When $z$ approaches the spectrum
of the Laplacian $\sigma(-\Delta)=[0,+\infty)$, it is possible
to define two limit operators
\begin{equation*}
  R_{0}(\lambda\pm i 0)=
  \lim_{\epsilon \downarrow0}R_{0}(\lambda\pm i \epsilon),
  \qquad
  \epsilon>0, \lambda\ge0
\end{equation*}
but the two limits are different if $\lambda>0$. These limits
exist in the norm of bounded operators
from the weighted $L^{2}_{s}$ space with norm
$\|\bra{x}^{s}f\|_{L^{2}}$ to the weighted Sobolev space
$H^{2}_{-s'}$ with norm 
$\sum_{|\alpha|\le2}\|\bra{x}^{-s'}\partial^{\alpha}f\|_{L^{2}}$,
for arbitrary $s,s'>1/2$ (see \cite{Agmon75-a}).
Since these spaces are dense in $\dot Y^{*}$ and
$\dot Y$ (or $\dot X$) respectively, and estimate
\eqref{eq:freeest} is uniform in $z$, one obtains that
\eqref{eq:freeest} is valid also for the limit operators
$R_{0}(\lambda\pm i0)$. In the following we shall write
simply $R_{0}(z)$, $z\in \overline{\mathbb{C}^{\pm}}$, 
to denote either one of 
the extended operators $R_{0}(\lambda\pm i \epsilon)$ with
$\epsilon\ge0$, defined on the closed upper (resp.~lower) complex
half--plane. Note also that the map $z \mapsto R_{0}(z)$
is continuous with respect to the operator norm of
bounded operators $L^{2}_{s}\to H^{2}_{-s'}$, 
for every $s,s'>1/2$, and from this fact one easily
obtains that it is also continuous with respect to
the operator norm of bounded operators
from $\dot Y^{*}\to H^{2}_{-s'}$. 

Thus in particular
\begin{equation*}
  R_{0}(z):\dot Y^{*}\to \dot X,
  \qquad
  \partial R_{0}(z):\dot Y^{*}\to \dot Y
\end{equation*}
are uniformly bounded operators for all
$z\in \overline{\mathbb{C}^{\pm}}$;
note also the formula
\begin{equation*}
  \Delta R_{0}(z)=-I-z R_{0}(z).
\end{equation*}
Moreover, for any smooth cutoff
$\phi\in C^{\infty}_{c}(\mathbb{R}^{n})$ and all 
$z\in \overline{\mathbb{C}^{\pm}}$,
the map 
$z\mapsto\phi R_{0}(z)$ is 
continuous w.r.to the norm of bounded operators
$\dot Y^{*}\to H^{2}$,
and hence 
\begin{equation*}
  \phi R_{0}(z):\dot Y^{*}\to L^{2}
  \quad\text{and}\quad 
  \phi \partial R_{0}(z):\dot Y^{*}\to L^{2}
  \quad \text{are compact operators.}
\end{equation*}
Similarly one gets that
$z\mapsto\phi R_{0}(z)$ is 
continuous w.r.to the norm of bounded operators
$\dot Y^{*}\to L^{\infty}_{|x|} L^{2}_{\omega}$ and
\begin{equation*}
  \phi R_{0}(z):\dot Y^{*}\to L^{\infty}_{|x|} L^{2}_{\omega}
  \quad \text{is a compact operator.}
\end{equation*}

In order to invert the operator $I-K(z)$
we shall apply Fredholm theory. The essential step is the
following compactness result:

\begin{lemma}[]\label{lem:compact}
  Let $z\in \overline{\mathbb{C}^{\pm}}$ and assume
  $W,A$ satisfy
  \begin{equation}\label{eq:assWZ}
    N:=
    \||x|^{3/2}(W+i( \partial \cdot A))\|_{\ell^{1}L^{2}L^{\infty}}
    +
    \||x|A\|_{\ell^{1}L^{\infty}}<\infty.
  \end{equation}
  Then
  $K(z)=(W+iA \cdot \partial+i \partial \cdot A)R_{0}(z)$
  is a compact operator on $\dot Y^{*}$,
  and the map $z \mapsto K(z)$ is continuous with respect
  to the norm of bounded operators on $\dot Y^{*}$.
\end{lemma}

\begin{proof}
  We decompose $K$
  as follows. Let $\chi\in C^{\infty}_{c}(\mathbb{R}^{n})$
  be a cutoff function equal to 1 for $|x|\le1$ and to 0 for
  $|x|\ge2$. Define for $r>2$
  \begin{equation*}
    \chi_{r}(x)=\chi(x/r)(1-\chi(rx))
  \end{equation*}
  so that $\chi_{r}$ vanishes for $|x|\ge2r$ and also for
  $|x|\le1/r$, and equals 1 when $2/r\le|x|\le r$.
  Then we split
  \begin{equation*}
    K=A_{r}+B_{r}
  \end{equation*}
  where
  \begin{equation*}
    A_{r}(z)=\chi_{r}\cdot K(z),
    \qquad
    B_{r}(z)=(1-\chi_{r})\cdot K(z).
  \end{equation*}

  First we 
  show that $A_{r}$ is a compact operator on $\dot Y^{*}$.
  Indeed, for $s>2r>4$ we have $\chi_{r}\chi_{s}=\chi_{r}$ and
  we can write
  \begin{equation*}
    A_{r}=\chi_{s}A_{r}=
    \chi_{s}(W+i( \partial \cdot A))\chi_{r}R_{0}(z)
    +
    2i \chi_{s}A \cdot \chi_{r}\partial R_{0}(z).
  \end{equation*}
  By the estimate
  \begin{equation}\label{eq:estVN}
    \|(W+i( \partial \cdot A))v\|_{\dot Y^{*}}\le
    \||x|^{3/2}(W+i( \partial \cdot A))\|_{\ell^{1}L^{2}L^{\infty}}
    \|v\|_{\dot X}
    \le
    N\|v\|_{\dot X}
  \end{equation}
  we see that multiplication by $W+i( \partial \cdot A)$ is a bounded operator
  from $\dot X$ to $\dot Y^{*}$.
  Moreover, multiplication by $\chi_{s}$ is a bounded
  operator $L^{\infty}_{|x|}L^{2}_{\omega}\to \dot X$
  and the operator
  $\chi_{r}R_{0}:\dot Y^{*}\to L^{\infty}_{|x|}L^{2}_{\omega}$
  is compact as remarked above. 
  A similar argument applies to the second term in $A_{r}$,
  using the estimate
  \begin{equation}\label{eq:estZN}
    \|Av\|_{\dot Y^{*}}
    \le
    \||x|A\|_{\ell^{1}L^{\infty}}\|v\|_{\dot Y}
    \le N\|v\|_{\dot Y}
  \end{equation}
  and compactness of $\chi_{r}\partial R_{0}:\dot Y^{*}\to L^{2}$.
  Summing up, we obtain that $A_{r}:\dot Y^{*}\to \dot Y^{*}$
  is a compact operator.
  Similarly, we see that $z \mapsto A_{r}(z)$ is continuous
  with respect to the norm of bounded operators on $\dot Y^{*}$.

  Then to conclude the proof it is sufficient to show that
  $B_{r}\to0$ in the norm of bounded operators
  on $\dot Y^{*}$, uniformly in $z$, as $r\to \infty$.
  We have, as in
  \eqref{eq:estVN}--\eqref{eq:estZN},
  \begin{equation*}
    \|B_{r}v\|_{\dot Y^{*}}
    \le
    N_{r}
    (\|R_{0}\|_{\dot Y^{*}\to \dot X}+
    \|\partial R_{0}\|_{\dot Y^{*}\to \dot Y})
    \|v\|_{\dot Y^{*}}
  \end{equation*}
  where
  \begin{equation*}
    N_{r}:=
    \||x|^{3/2}(1-\chi_{r})(W+i( \partial \cdot A))
      \|_{\ell^{1}L^{2} L^{\infty}}
    +
    2\||x|(1-\chi_{r})A\|_{\ell^{1}L^{\infty}}.
  \end{equation*}
  Since $N_{r}\to0$ as $r\to \infty$,
  we obtain that $\|B_{r}\|_{\dot Y^{*}\to \dot Y^{*}}\to0$.
\end{proof}

We now study the injectivity of $I-K(z):\dot Y^{*}\to \dot Y^{*}$.
Note that if $f\in \dot Y^{*}$ satisfies
\begin{equation*}
  (I-K(z))f=0
\end{equation*}
then setting $v=R_{0}(z)$ 
by the properties of $R_{0}(z)$ we have
$v\in H^{1}_{loc} \cap \dot X$, 
$\nabla v\in \dot Y$,
$v\in H^{2}_{loc}(\mathbb{R}^{n}\setminus0)$,
$\Delta v\in \dot Y+\dot Y^{*}$
(or $\Delta v\in\dot Y^{*}$ if $z=0$)
and if $z\neq0$ we have also $v\in\dot Y$.
In particular, $v$ is a solution of the
equation 
\begin{equation*}
  (H-z)v=0.
\end{equation*}
For $z$ outside the spectrum of $H$ it is easy to check
that this implies $v=f=0$:

\begin{lemma}[]\label{lem:eigvcomplex}
  Let $W,A,K(z)$ be as in Lemma \ref{lem:compact} and
  $H=-\Delta-W-iA \cdot \partial-i \partial \cdot A$.
  If $f\in \dot Y^{*}$ satisfies
  \begin{equation*}
    (I-K(z))f=0
  \end{equation*}
  for some $z\not\in \sigma(H)$, then $f=0$.
\end{lemma}

\begin{proof}
  Let $v=R_{0}(z)f$,
  fix a compactly supported smooth function $\chi$ which is
  equal to 1 for $|x|\le1$, and for $M>1$ consider
  $v_{M}:=v(x)\chi(x/M)$. Then $v_{M}\in L^{2}$ and
  \begin{equation*}
    \textstyle
    (H-z)v_{M}=
    \frac1M \nabla\chi(\frac xM)(2\nabla v+2iA v)
    +\frac1{M^{2}} \Delta \chi(\frac xM)v
    =:f_{M}.
  \end{equation*}
  We have, for $\delta\in(1,\frac12)$,
  using the estimate $|A|\lesssim|x|^{-1}$,
  \begin{equation*}
  \begin{split}
    \|f_{M}\|_{L^{2}}
    \lesssim
    &
    M^{\delta-\frac12}
    \left(
      \||x|^{-\frac12-\delta}\nabla v\|_{L^{2}(|x|\ge M)}
      +
      \||x|^{-\frac32-\delta}v\|_{L^{2}(|x|\ge M)}
    \right)
    \\
    \lesssim
    &
    M^{\delta-\frac12}
    (\|\nabla v\|_{\dot Y}+\|v\|_{\dot X})
    \end{split}
  \end{equation*}
  uniformly in $M$, so that $f_{M}\to 0$ in $L^{2}$
  as $M\to \infty$.
  Since $v_{M}=R_{0}(z)f_{M}$ and $R_{0}(z)$ is a bounded
  operator on $L^{2}$, we conclude that $v=f=0$.
\end{proof}

The hard case is of course $z\in \sigma(L)$. Then we have the
following result, in which we write simply
\begin{equation*}
  R_{0}(\lambda)
  \quad\text{instead of}\quad 
  R_{0}(\lambda\pm i0)
\end{equation*}
since the computations for the two cases are identical.
Note that this is the only step where we use the 
additional $\delta$ decay of the coefficients.

\begin{lemma}[]\label{lem:eigv}
  Assume $W$ and $A$ satisfy for some $\delta>0$
  \begin{equation}\label{eq:WZdelta}
    |x|^{2}\bra{x}^{\delta}(W+i\partial \cdot A)
    \in \ell^{1}L^{\infty},
    \qquad
    |x|\bra{x}^{\delta}A
    \in \ell^{1}L^{\infty}
  \end{equation}
  and 
  $H=-\Delta -W-i \partial \cdot A-iA \cdot \partial $ 
  is a non negative selfadjoint operator on $L^{2}$.
  Let $f\in \dot Y^{*}$ be such that, for some
  $\lambda\ge0$,
  \begin{equation*}
    (I-K(\lambda))f=0,
    \qquad
    K(\lambda)
    :=(W+i \partial \cdot A+iA \cdot\partial)R_{0}(\lambda).
  \end{equation*}
  Then in the case $\lambda>0$ we have $f=0$,
  while in the case $\lambda=0$ we have
  $|x|^{n/2}f\in L^{2}$ and the function $v=R_{0}(0)f$
  belongs to $H^{2}_{loc}(\mathbb{R}^{n}\setminus0)\cap \dot X$
  with $\partial v\in \dot Y$,
  solves $Lv=0$ and satisfies 
  $|x|^{\frac n2-2-\delta'}v\in L^{2}$  and
  $|x|^{\frac n2-1-\delta'}\partial v\in L^{2}$ 
  for any $\delta'>0$.
\end{lemma}

\begin{proof}
  Defining as in the previous proof $v=R_{0}(\lambda)f$,
  we see that $v$ solves
  \begin{equation}\label{eq:eqvg}
    \Delta v+\lambda v+g=0,
    \qquad
    g:=Wv+i A \cdot \partial v +i \partial \cdot Av.
  \end{equation}
  Then given a radial function $\chi\ge0$ to be precised later, we
  apply again identity \eqref{eq:fundid}
  with the choices
  \begin{equation*}
    \psi'=\chi,
    \qquad
    \phi=-\chi'
  \end{equation*}
  so that in particular $\Delta \psi+\phi=\frac{n-1}{|x|}\chi$.
  We integrate the identity on $\mathbb{R}^{n}$ and,
  after straightforward computations
  (see Proposition 3.1 of 
  \cite{CacciafestaDAnconaLuca16-b}
  for a similar argument), we arrive at the
  following \emph{radiation estimate}:
  \begin{equation}\label{eq:radest}
    \textstyle
    \int \chi'|\partial_{S}v|^{2}
    +
    2(\frac{\chi}{|x|}-\chi')|(\partial v)_{T}|^{2}
    -
    \frac{n-1}2\int \Delta(\frac{\chi}{|x|})|v|^{2}
    =
    \Re \int\chi g(\frac{n-1}{|x|}\overline{v}+
      2 \widehat{x}\cdot\overline{\partial_{S}v})
  \end{equation}
  where we denoted the "Sommerfeld" gradient of $v$ with
  \begin{equation*}
    \partial_{S}v:=\partial v-i \sqrt{\lambda}\widehat{x}v,
    \qquad
    \widehat{x}=x/|x|
  \end{equation*}
  and the tangential component of $\partial v$ with
  \begin{equation*}
    |(\partial v)_{T}|^{2}:=
    |\partial v|^{2}-|\widehat{x}\cdot \partial v|^{2}.
  \end{equation*}

  We now estimate the right hand side of \eqref{eq:radest}.
  We have
  \begin{equation*}
    \textstyle
    |\Re\int\chi g \frac{\overline{v}}{|x|}|
    \le
    \|\chi(W+i (\partial \cdot A))|x|^{-1}|v|^{2}\|_{L^{1}}
    +
    2\|\chi A|\partial v| |x|^{-1}v\|_{L^{1}}
  \end{equation*}
  \begin{equation*}
    \le
    \|\chi|x|(W+i(\partial \cdot A))\|
      _{\ell^{1}L^{1}L^{2}}
    \|v\|_{\dot X}^{2}
    +
    2\|\chi|x|^{1/2}A\|_{\ell^{1}L^{2}L^{\infty}}
    \|v\|_{\dot X}\|\nabla v\|_{\dot Y}
  \end{equation*}
  and similarly
  \begin{equation*}
    \textstyle
    |\Re\int \chi g \widehat{x}\cdot\overline{\partial_{S}v}|
    \le
    \|\chi(W+i(\partial \cdot A)) v|\partial_{S}v|\|_{L^{1}}
    +
    \|\chi A \cdot \partial v |\partial_{S}v|\|_{L^{1}}
  \end{equation*}
  \begin{equation*}
    \le
    \|\chi|x|^{3/2}(W+i(\partial \cdot A))\|
      _{\ell^{1}L^{2}L^{\infty}}
    \|v\|_{\dot X}\|\partial_{S}v\|_{\dot Y}
    +
    \|\chi|x|A\|_{\ell^{1}L^{\infty}}
    \|\partial v\|_{\dot Y}\|\partial_{S}v\|_{\dot Y}.
  \end{equation*}
  Since the quantities $\|v\|_{\dot X}$, $\|\partial v\|_{\dot Y}$
  and 
  $\|\partial_{S}v\|_{\dot Y}\le\|\partial v\|_{\dot Y}+
  \sqrt{\lambda}\|v\|_{\dot Y}$ are all estimated by
  $\|f\|_{\dot Y^{*}}$ (recall \eqref{eq:freeest}),
  we conclude
  \begin{equation}\label{eq:estgsmall}
    \textstyle
    \left|\Re \int\chi g(\frac{n-1}{|x|}\overline{v}+
      2 \widehat{x}\cdot\overline{\partial_{S}v})\right|
    \lesssim
    N_{\chi}^{2}\|f\|_{\dot Y^{*}}^{2}
  \end{equation}
  where
  \begin{equation*}
    N_{\chi}^{2}:=
    \|\chi|x|^{3/2}(W+i(\partial \cdot A))\|
      _{\ell^{1}L^{2}L^{\infty}}
    +
    \|\chi|x|A\|_{\ell^{1}L^{\infty}}.
  \end{equation*}
  Finally, if we choose 
  \begin{equation*}
    \chi(x)=|x|^{\delta}
    \quad\text{with}\quad 
    0<\delta\le 1
  \end{equation*}
  by \eqref{eq:radest} and \eqref{eq:estgsmall} we obtain,
  dropping a (nonnegative) term at the left,
  \begin{equation}\label{eq:finaldelta}
    \||x|^{(\delta-1)/2}\partial_{S}v\|_{L^{2}}
    +
    \||x|^{(\delta-3)/2}v\|_{L^{2}}
    \lesssim_{\delta}
    N_{\delta}\|f\|_{\dot Y^{*}}
  \end{equation}
  where by assumption
  \begin{equation*}
    N_{\delta}^{2}
    :=
    \||x|^{3/2+\delta}(W+i(\partial \cdot A))\|
      _{\ell^{1}L^{2}L^{\infty}}
    +
    \||x|^{1+\delta}A\|_{\ell^{1}L^{\infty}}<\infty.
  \end{equation*}

  Consider now the following identity, obtained using the
  divergence formula:
  \begin{equation*}
    \textstyle
    \int_{|x|= R}
    (|\partial v|^{2}+\lambda|v|^{2}-|\partial_{S}v|^{2})
    d \sigma
    =
    2\Re\int_{|x|\le R} i\sqrt{\lambda}\,
    \partial \cdot ( v\,\overline{\partial v})
    =
    2\Re\int_{|x|\le R} i\sqrt{\lambda}\,
    (v\,\overline{\Delta v})
  \end{equation*}
  for arbitrary $R>0$. Substituting
  $\Delta v=-\lambda v-g$ from \eqref{eq:eqvg}
  and dropping two pure imaginary terms, we get
  \begin{equation*}
    \textstyle
    \int_{|x|= R}
    (|\partial v|^{2}+\lambda|v|^{2}-|\partial_{S}v|^{2})
    d \sigma
    =
    2\Re
    \int_{|x|\le R}
    (A \cdot \partial v +\partial \cdot Av)\overline{v}.
  \end{equation*}
  The last term can be written,
  again by the divergence formula,
  \begin{equation*}
    \textstyle
    =2\int_{|x|\le R}\partial \cdot(A|v|^{2})
    =2\sum_{j}\int_{|x|=R}
    \widehat{x}_{j}
    \cdot A|v|^{2}
    d \sigma,
    \qquad
    \widehat{x}_{j}=x_{j}/|x|.
  \end{equation*}
  By assumption $|A|\lesssim|x|^{-1}$, hence for some $R_{0}>0$
  we have $\lambda>2|A(x)|$ for all $|x|>R_{0}$, and the term in 
  $A$ can be absorbed at the left of the identity. Summing up,
  we have proved that
  \begin{equation}\label{eq:equivna}
    \textstyle
    \int_{|x|=R}(|\partial v|^{2}+\lambda|v|^{2})d \sigma
    \le
    2\int_{|x|=R}|\partial_{S} v|^{2}d \sigma,
    \qquad
    R\ge R_{0}.
  \end{equation}
  Multiplying both sides by $|x|^{\delta-1}$,
  integrating in the radial direction from $R_{0}$ to $\infty$,
  and using \eqref{eq:finaldelta}, we conclude
  \begin{equation}\label{eq:deltaL2}
    \||x|^{(\delta-1)/2}\partial v\|_{L^{2}(|x|\ge R_{0})}
    +
    \sqrt{\lambda}
    \||x|^{(\delta-1)/2} v\|_{L^{2}(|x|\ge R_{0})}
    \lesssim
    \|f\|_{\dot Y^{*}}.
  \end{equation}

  In the case $\lambda>0$ we have proved that
  $|x|^{(\delta-1)/2}v\in L^{2}$ i.e., $\lambda$ is a resonance,
  and this is enough to conclude that $v=0$ by applying
  one of the available results on the absence of embedded
  eigenvalues. We shall apply
  the results from \cite{KochTataru06-a} which are 
  partiularly sharp.
  We need to check the assumptions on the potentials required
  in \cite{KochTataru06-a}.
  The potential $V$ in \cite{KochTataru06-a} is simply $V=z$ 
  in our case, which we are assuming real and $>0$, 
  thus condition A.1 is trivially satisfied.
  Concerning $W$ we have
  \begin{equation*}
    \|W\|_{L^{n/2}}
    \le
    \||x|^{-2}\|_{\ell^{\infty}L^{n/2}}
    \||x|^{2}W\|_{\ell^{n/2}L^{\infty}}<\infty
  \end{equation*}
  by assumption,
  thus $W\in L^{n/2}$ and condition A.2 in
  \cite{KochTataru06-a} is satisfied
  Concerning the potential $Z$ in the notations
  of \cite{KochTataru06-a}, which coincides with
  $A$ here, we have
  \begin{equation*}
    \|A\|_{\ell^{\infty}L^{n}}
    \le
    \||x|^{-1}\|_{\ell^{\infty}L^{n}}
    \||x|A\|_{\ell^{n} L^{\infty}}<\infty
  \end{equation*}
  thus $A\in \ell^{\infty}L^{n}$; moreover a similar computation
  applied to $\one{|x|>M}A$ gives
  \begin{equation*}
    \|\one{|x|>M}A\|_{\ell^{\infty}L^{n}}
    \le
    \||x|^{-1}\|_{\ell^{\infty}L^{n}}
    \|\one{|x|>M}|x|^{-\delta}\|_{L^{\infty}}
    \||x|^{1+\delta}A\|_{\ell^{n} L^{\infty}}\to0
    \quad\text{as}\quad M\to \infty.
  \end{equation*}
  Thus to check that $A$ satisfies condition A.3
  in \cite{KochTataru06-a} it remains to check that
  the low frequency part $S_{<R}A$ of $A$ satisfies A.2
  for $R$ large enough. $S_{<R}A$ is obviously smooth.
  Moreover, it is clear that
  $|x|A\to0$ as $|x|\to \infty$;
  in order to prove the same
  decay property for $S_{<R}A$ we represent it as
  a convolution with a suitable Schwartz kernel $\phi$
  \begin{equation*}
    \textstyle
    \phi *A(x)=
    \int_{|y|\le \frac{|x|}{2}}
    A(y)\phi(x-y)
    +
    \int_{|y|\ge \frac{|x|}{2}}
    A(y)\phi(x-y).
  \end{equation*}
  The first integral is bounded by $C_{k} \bra{x}^{-k}$ for all $k$.
  For the second one we write
  \begin{equation*}
    \textstyle
    |x|\int_{|y|\ge \frac{|x|}{2}}
    A(y)\phi(x-y)
    \le
    \int_{|y|\ge \frac{|x|}{2}}
    |y|A(y)\phi(x-y)=o(|x|).
  \end{equation*}
  We have thus proved that $|x|S_{<R}A\to0$ as $|x|\to \infty$
  (for any fixed $R$) and hence $A=Z$ satisfies condition A.3.
  Applying Theorem 8 of \cite{KochTataru06-a}, we conclude
  that $v=0$.

  It remains to consider the case $\lambda=0$.
  We denote by $\dot L^{2}_{s}$ the Hilbert space with norm
  \begin{equation*}
    \|v\|_{\dot L^{2}_{s}}:=\||x|^{s}v\|_{L^{2}}.
  \end{equation*}
  By the well known Stein--Weiss estimate for fractional
  integrals in weighted $L^{p}$ spaces, applied to
  $R_{0}(0)v=\Delta^{-1}v=c|x|^{2-n}*v$, we see that $R_{0}(0)$
  is a bounded operator
  \begin{equation*}
    \textstyle
    R_{0}(0):\dot L^{2}_{s}\to \dot L^{2}_{s-2}
    \qquad\text{for all}\qquad
    2-\frac n2<s<\frac n2
  \end{equation*}
  while $\partial R_{0}(0)=c (x|x|^{-n})*v$ is a bounded operator
  \begin{equation*}
    \textstyle
    \partial R_{0}(0):\dot L^{2}_{s}\to \dot L^{2}_{s-1}
    \qquad\text{for all}\qquad
    1-\frac n2<s<\frac n2.
  \end{equation*}
  Recall also that $R_{0}(0)$ is bounded from $\dot Y^{*}$
  to $\dot X$ and $\partial R_{0}(0)$ is bounded from
  $\dot Y^{*}$ to $\dot Y$.
  Moreover from the assumption on $W,A$ it follows that
  the corresponding multiplication operators are bounded
  operators
  \begin{equation*}
    \textstyle
    W+i (\partial \cdot A):\dot X\to \dot L^{2}_{1/2+\delta},
    \qquad
    W+i (\partial \cdot A):\dot L^{2}_{s-2}\to \dot L^{2}_{s+\delta}
    \qquad \forall s\in \mathbb{R},
  \end{equation*}
  \begin{equation*}
    A:\dot Y\to \dot L^{2}_{1/2+\delta},
    \qquad
    A:\dot L^{2}_{s-1}\to \dot L^{2}_{s+\delta}
    \qquad \forall s\in \mathbb{R}.
  \end{equation*}
  Conbining all the previous properties we deduce that
  $K(0)=(W+i \partial \cdot A+i A \cdot \partial)R_{0}(0)$
  is a bounded operator
  \begin{equation}\label{eq:Kbdd}
    \textstyle
    K(0):\dot Y^{*}\to \dot L^{2}_{1/2+\delta}
    \quad\text{and}\quad 
    K(0):\dot L^{2}_{s}\to \dot L^{2}_{s+\delta},
    \qquad
    \forall\ 2-\frac n2<s<\frac n2.
  \end{equation}
  Since we know that $f\in \dot Y^{*}$ and that
  $f=K(0)f$, applying \eqref{eq:Kbdd} repeatedly,
  we obtain in a finite number of steps that 
  $f\in \dot L^{2}_{n/2}$, which in turn implies
  $v=R_{0}(0)f\in \dot L^{2}_{s}$ for all $s<\frac n2-2$
  and
  $\partial v=\partial R_{0}(0)f\in \dot L^{2}_{s}$ for all 
  $s<\frac n2-1$. The proof is concluded.
\end{proof}

If $K(z)$ is compact and $I-K(z)$ is injective on $\dot Y^{*}$
(under suitable assumptions),
it follows from Fredholm theory that $(I-K(z))^{-1}$ is a
bounded operator for all $z\in \mathbb{C}$. However we need
a bound uniform in $z$, and to this end it is sufficient
to prove that the map $z \mapsto(I-K(z))^{-1}$
is continuous. 
This follows from a general well known
result which we reprove here for the benefit of the reader.
Note that $z \mapsto I-K(z)$ is trivially 
continuous (and holomorphic for $z\not\in \sigma(H)$).

\begin{lemma}[]\label{lem:invertImK}
  Let $X_{1},X_{2}$ be two Banach spaces,
  $K_{j},K$ compact operators from $X_{1}$ to $X_{2}$,
  and assume the sequence
  $K_{j}\to K$ in the operator norm as $j\to \infty$.
  If $I-K_{j}$, $I-K$ are invertible with bounded inverses,
  then $(I-K_{j})^{-1}\to(I-K)^{-1}$ in the operator norm.
\end{lemma}

\begin{proof}
  Let $\phi\in X_{2}$ and let
  $c_{j}:=\|(I-K_{j})^{-1}\phi\|_{X_{1}}$.
  If by contradiction $c_{j}\to \infty$, then
  defining $\psi_{j}=(I-K_{j})^{-1}\phi \cdot c_{j}^{-1}$
  and $\phi_{j}=\phi \cdot c_{j}^{-1}$ we would have
  \begin{equation*}
    \|\psi_{j}\|_{X_{1}}=1,
    \qquad
    \|\phi_{j}\|_{X_{2}}\to 0,
    \qquad
    \phi_{j}=(I-K_{j})\psi_{j}.
  \end{equation*}
  The last identity can be written
  \begin{equation*}
    \psi_{j}=\phi_{j}+(K_{j}-K)\psi_{j}+K \psi_{j}.
  \end{equation*}
  The first two terms at the right tend to 0, and the third
  one converges, by possibly passing to a subsequence, since
  $K$ is compact; let $\psi=\lim K \psi_{j}$. By the previous
  identity we see that also $\psi_{j}$ converges to $\psi$
  so that $\|\psi\|=1$ and $\psi=K \psi$, which contradicts 
  the invertibility of $I-K$. 

  We have thus proved that, for any $\phi\in X_{2}$, the sequence
  $\chi_{j}:=(I-K_{j})^{-1}\phi$ is bounded in $X_{1}$.
  Write this identity in the form
  \begin{equation*}
    \chi_{j}=\phi+K \chi_{j}+(K_{j}-K)\chi_{j}
  \end{equation*}
  and note as before that $K \chi_{j}$ is a relatively compact
  sequence; let $\chi$ be any one of its limit points.
  Letting $j\to \infty$ we get $\chi=\phi+K \chi$, i.e.,
  $(I-K_{j})^{-1}\phi\to(I-K)^{-1}\phi$.
  Applying the uniform boundedness principle we get the claim.
\end{proof}

We finally sum up the previous results. 
We shall need to assume that
$0$ \emph{is not a resonance}, in the sense
of Definition \ref{def:reson}.
Note that in 
Lemma \ref{lem:eigv} we proved in particular that if 
$f\in \dot Y^{*}$ satisfies $f=K(0)f$, then $v=R_{0}(0)f$
is a resonant state at 0.

\begin{theorem}[]\label{the:invert}
  Assume the operator $H$ defined in \eqref{eq:defL}
  is non negative and selfadjoint on $L^{2}$, with
  $W$ and $A$ satisfying \eqref{eq:WZdelta}
  for some $\delta>0$. In addition, asssume that
  $0$ is not a resonance for $H$, in the sense of
  Definition \ref{def:reson}.

  Then $I-K(z)$ is a bounded invertible operator 
  on $\dot Y^{*}$, with $(I-K(z))^{-1}$
  bounded uniformly for $z$ in bounded subsets of 
  $\overline{\mathbb{C}^{\pm}}$.
  Moreover, the resolvent operator $R(z)=(H-z)^{-1}$ 
  satisfies the estimate
  \begin{equation}\label{eq:resestim}
    \|R(z)f\|_{\dot X}
    +
    |z|^{\frac12}\|R(z)f\|_{\dot Y}
    +
    \|\partial R(z)f\|_{\dot Y}
    \le
    C(z)
    \|f\|_{\dot Y^{*}}
  \end{equation}
  for all $z\in\overline{\mathbb{C}^{\pm}}$, where $C(z)$
  is a continuous function of $z$.
\end{theorem}

\begin{proof}
  It is sufficient to combine
  Lemmas \ref{lem:compact},
  \ref{lem:eigvcomplex},
  \ref{lem:eigv},
  \ref{lem:invertImK}
  and apply Fredholm theory in conjuction with assumption
  \eqref{eq:resona}, to prove the claims about
  $I-K(z)$; note that \eqref{eq:WZdelta} include
  the assumptions of Lemmas \ref{lem:compact}--\ref{lem:invertImK}.
  Finally, using the representation \eqref{eq:LSE}
  and the free estimate \eqref{eq:freeest} we 
  obtain \eqref{eq:resestim}.
\end{proof}

\section{The full resolvent estimate}\label{sec:the_full_reso_esti}

In this Section and the following ones we shall freely 
use a few
results from classical harmonic analysis, in particular
the basic properties of Muckenhoupt classes $A_{p}$
and Lorentz spaces. For more details
see e.g. \cite{Grafakos08-a}, \cite{JohnsonNeugebauer91-a} and
\cite{Stein93-a}.

Consider the operator $L$ defined by
\begin{equation}\label{eq:opL}
  Lv=-\Delta_{A}v+Vv
\end{equation}
and the resolvent equation
\begin{equation}\label{eq:reseqL}
  Lv-zv=f.
\end{equation}
We put together the estimates of the previous Sections
to obtain:

\begin{theorem}[Resolvent estimate]\label{the:fullresest}
  Let $n\ge3$.
  Assume the operator $L$ defined in \eqref{eq:opL} is selfadjoint
  and non negative on $L^{2}(\mathbb{R}^{n})$, with domain
  $H^{2}(\mathbb{R}^{n})$.
  Assume $V:\mathbb{R}^{n}\to \mathbb{R}$
  and $A:\mathbb{R}^{n}\to \mathbb{R}^{n}$ satisfy
  for some $\delta>0$:
  \begin{equation}\label{eq:assVA}
    |x|^{2}\bra{x}^{\delta}(V-i \partial \cdot A),
    \quad
    |x|\bra{x}^{\delta}A
    \quad\text{and}\quad 
    |x|\bra{x}^{\delta}\widehat{B}
    \quad\text{belong to}\quad 
    \ell^{1}L^{\infty}.
  \end{equation}
  Moreover, assume 0 is not a resonance for $L$,
  in the sense of Definition \ref{def:reson}.

  Then for all $z\in \overline{\mathbb{C}^{\pm}}$ with
  $|\Im z|\le1$ the resolvent operator $R(z)=(L-z)^{-1}$
  satisfies the following estimate uniform in $z$:
  \begin{equation}\label{eq:resestim2}
    \|R(z)f\|_{\dot X}
    +
    |z|^{\frac12}\|R(z)f\|_{\dot Y}
    +
    \|\partial R(z)f\|_{\dot Y}
    \lesssim
    \|f\|_{\dot Y^{*}}.
  \end{equation}
\end{theorem}

\begin{proof}
  The proof is obtained by combining the estimates of 
  Theorems \ref{the:resestlarge} and \ref{the:invert}.
  In order to apply Theorem \ref{the:invert}, we write $L$
  in the form
  \begin{equation*}
    Lv=-\Delta v +(V+|A|^{2})
      -iA \cdot \partial v-i \partial \cdot(Av)
  \end{equation*}
  which coincides with $H$ defined in \eqref{eq:defL}
  with the choice $W=-V-|A|^{2}$.
  The assumptions of Theorem \ref{the:invert},
  see \eqref{eq:WZdelta}, are satisfied if
  \begin{equation*}
    |x|^{2}\bra{x}^{\delta}(V+|A|^{2}-i \partial \cdot A)
    \quad\text{and}\quad 
    |x|\bra{x}^{\delta}A
    \quad\text{belong to}\quad 
    \ell^{1}L^{\infty}
  \end{equation*}
  and these conditions are implied by \eqref{eq:assVA},
  with a possibly different $\delta$.
  This proves \eqref{eq:resestim2} for $z$ in any bounded set.

  In order to apply Theorem \ref{the:resestlarge} we note that
  the operato $L$ is already in the form required
  for \eqref{eq:reseq}, choosing $Z=0$ and $W=-V$. We check
  assumption \eqref{eq:asscoeff}: the assumptions on
  $\widehat{B}$ (and $Z=0$) are satisfied. Next
  we split $W=-V$ as
  \begin{equation*}
    V_{r}=\one{|x|\le r}V,
    \qquad
    V'_{r}=V-V_{r}
  \end{equation*}
  and we note that from $|x|^{2}V\in \ell^{1}L^{\infty}$
  it follows that
  \begin{equation*}
    \||x|^{3/2}V_{r}\|_{\ell^{1}L^{2}L^{\infty}}
    \le
    \||x|^{2}V_{r}\|_{\ell^{1}L^{\infty}}
    \to0
    \quad\text{as}\quad 
    r\to0.
  \end{equation*}
  On the other hand, $|x|V'_{r}\in \ell^{1}L^{\infty}$
  for any $r$. Thus if we choose
  $W_{S}=-V_{r}$, $W_{L}=-V'_{r}$ for $r$ sufficiently small,
  then \eqref{eq:asscoeff} are satisfied. This proves
  \eqref{eq:resestim2} for all sufficiently large $z$
  belonging to the strip $|\Im z|\le1$, with a constant
  independent of $z$, and the proof is concluded.
\end{proof}

\begin{corollary}[]\label{cor:onehalfres}
  Under the assumptions of Theorem \ref{the:fullresest},
  for all $z\in \overline{\mathbb{C}^{\pm}}$ with
  $|\Im z|\le1$ the resolvent operator $R(z)=(L-z)^{-1}$
  satisfies the following estimate, uniform in $z$:
  \begin{equation}\label{eq:onehalfres}
    \||D|^{1/2}R(z)|D|^{1/2}f\|_{\dot Y}
    +
    \||x|^{-1/2}R(z)|x|^{-1/2}f\|_{\dot Y}
    \lesssim
    \|f\|_{\dot Y^{*}}.
  \end{equation}
\end{corollary}

\begin{proof}
  Recall that $|x|^{-s}$ is in the Muckenhoupt class
  $A_{2}$ if and only if $|s|<n$, and this implies that the
  Riesz operator 
  $\rie v:= \mathcal{F}^{-1}(\widehat{v}(\xi)\xi/|\xi|)$
  (where both $\mathcal{F}v$ and $\widehat{v}$ denote Fourier
  transform) satisfies the weighted estimate
  \begin{equation}\label{eq:estrie}
    \||x|^{-s}\rie v\|_{L^{2}}\lesssim\||x|^{-s}v\|_{L^{2}}
  \end{equation}
  for all $|s|<n/2$. Introduce the weighted dyadic norms
  \begin{equation*}
    \textstyle
    \|v\|_{\ell^{q}(2^{-js})L^{2}}:=
    \left\|2^{-js}\|v\|_{L^{2}(C_{j})}
    \right\|_{\ell^{q}_{j}}
    =
    \left(\sum_{j\in \mathbb{Z}}2^{-qjs}
    \|v\|_{L^{2}(C_{j})}^{q}\right)^{1/q}
  \end{equation*}
  where $C_{j}\subseteq \mathbb{R}^{n}$ is the ring 
  $2^{j}\le|x|<2^{j+1}$ as usual. Then \eqref{eq:estrie}
  can be written
  \begin{equation}\label{eq:seqrie}
    \textstyle
    \|\rie v\|_{\ell^{2}(2^{-js})L^{2}}
    \lesssim
    \|v\|_{\ell^{2}(2^{-js})L^{2}}
    \qquad
    \forall|s|<\frac n2.
  \end{equation}
  We recall now the real interpolation formula:
  if $q_{0},q_{1},q\in(0,\infty]$, $\theta\in(0,1)$,
  $s_{0}\neq s_{1}\in \mathbb{R}$,
  \begin{equation*}
    (\ell^{q_{0}}(2^{-js_{0}})L^{2},
    \ell^{q_{1}}(2^{-js_{1}})L^{2})_{\theta , q}
    \simeq
    \ell^{q}(2^{-js})L^{2},
    \qquad
    s=(1-\theta)s_{0}+\theta s_{1}
  \end{equation*}
  (Theorem 5.6.1 in \cite{BerghLofstrom76-a}).
  If we apply the formula with
  $q_{0}=q_{1}=2$, $q=\infty$, $s_{0}=1/2-\epsilon$, 
  $s_{1}=1/2+\epsilon$
  with $\epsilon>0$ and $\theta=1/2$, we obtain
  \begin{equation*}
    (\ell^{2}(2^{-j(1/2-\epsilon)})L^{2},
    \ell^{2}(2^{-j(1/2-\epsilon)})L^{2})_{1/2 , \infty}
    =
    \ell^{\infty}(2^{-j/2})L^{2}
    \simeq
    \dot Y.
  \end{equation*}
  Then, interpolating the inequalities \eqref{eq:seqrie} for
  $s=1/2\pm \epsilon$ with $\epsilon>0$ small, we
  obtain
  \begin{equation}\label{eq:rieY}
    \|\rie v\|_{\dot Y}\lesssim\|v\|_{\dot Y}
  \end{equation}
  i.e., the Riesz operator is bounded on $\dot Y$.
  By duality, $\rie$ is also bounded on $\dot Y^{*}$.

  Exactly the same argument applies to the 
  Calder\'{o}n--Zygmund operators
  $|D|^{iy}$, $y\in \mathbb{R}$, which are defined via
  the formula 
  $|D|^{iy}v:=\mathcal{F}^{-1}(|\xi|^{iy}\widehat{v}(\xi))$,
  thus we have for all $y\in \mathbb{R}$
  \begin{equation}\label{eq:DyY}
    \||D|^{iy}v\|_{\dot Y}\lesssim\|v\|_{\dot Y}.
  \end{equation}
  with a norm growing polynomially in $y\in \mathbb{R}$
  (like $|y|^{n/2}$ at most).
  The same property holds for $|D|^{iy}:\dot Y^{*}\to \dot Y^{*}$.

  Now we can write, by \eqref{eq:rieY} and \eqref{eq:resestim2},
  \begin{equation*}
    \||D|R(z)f\|_{\dot Y}=
    \|\rie \cdot\partial R(z)f\|_{\dot Y}
    \lesssim
    \|f\|_{\dot Y^{*}}
  \end{equation*}
  uniformly in $z$. Thus $|D|R(z):\dot Y^{*}\to \dot Y$
  is bounded, uniformly in $z$, and by duality the same holds
  for $R(z)|D|$.

  We now apply
  Stein--Weiss interpolation to the analytic family of operators
  \begin{equation*}
    T_{w}:=|D|^{1-w}R(z)|D|^{w},
    \qquad w\text{ in the complex strip } 
    0\le\Im w\le 1.
  \end{equation*}
  Indeed, writing
  \begin{equation*}
    T_{iy}=|D|^{-iy}|D|R(z)|D|^{iy},
    \qquad
    T_{1+iy}=|D|^{-iy}\cdot R(z)|D|\cdot |D|^{iy}
  \end{equation*}
  and using the previous steps, we see that
  $T_{w}:\dot Y^{*}\to \dot Y$ is a bounded operator
  for $\Re w=0$ and $\Re w=1$, uniformly in $y=\Im w$,
  which implies $T_{w}:\dot Y^{*}\to \dot Y$ is a bounded
  operator for all $w$ in the strip. Taking $w=1/2$
  we prove the first part of \eqref{eq:onehalfres}.

  Consider now the second part of \eqref{eq:onehalfres}.
  Recalling that $\||x|^{-1}v\|_{\dot Y}\le\|v\|_{\dot X}$,
  from \eqref{eq:resestim2} we have in particular
  \begin{equation*}
    \||x|^{-1}R(z)f\|_{\dot Y}\lesssim\|f\|_{\dot Y^{*}}
  \end{equation*}
  and hence by duality
  \begin{equation*}
    \|R(z)|x|^{-1}f\|_{\dot Y}\lesssim\|f\|_{\dot Y^{*}}.
  \end{equation*}
  Interpolating between these estimates as in the first
  part of the proof, we obtain \eqref{eq:onehalfres}.
\end{proof}

\begin{remark}[]\label{rem:delta0}
  The weight $\bra{x}^{\delta}$ with $\delta>0$
  in assumption \eqref{eq:assVA} is required only to
  exclude resonances embedded or at the treshold, using
  Lemma \ref{lem:eigv}.
  If we assume a priori the condition
  \begin{equation}\label{eq:noreson}
    (L+\lambda)v=0,
    \qquad
    v\in H^{2}_{loc}\cap \dot Y,
    \qquad
    \lambda\ge0
    \quad\implies\quad
    v=0
  \end{equation}
  then Lemma \ref{lem:eigv} is no longer necessary
  and Theorem \ref{the:fullresest} holds with
  $\delta=0$.
\end{remark}

\begin{remark}[Gauge transformation]\label{rem:gauge}
  If we apply a change of gauge
  \begin{equation*}
    u=e^{i \phi(x)}v
  \end{equation*}
  the magnetic Laplacian transforms as follows:
  \begin{equation}\label{eq:gaugefor}
    \Delta_{A}(e^{i \phi}v)=e^{i \phi}\Delta_{\widetilde{A}}v,
    \qquad
    \widetilde{A}=A+\partial \phi.
  \end{equation}
  In particular, if we choose
  \begin{equation*}
    \phi(x)=\Delta^{-1}\partial \cdot A
    \implies
    \partial \cdot \widetilde{A}=0
  \end{equation*}
  we see that we can gauge away the term $\partial \cdot A$
  with an appropriate choice of $\phi$
  in Theorem \ref{the:fullresest},
  although the details require some work.
  Note also that the magnetic field $B$ is \emph{gauge invariant},
  since 
  $\partial_{j}\partial_{k}\phi-\partial_{k}\partial_{j}\phi=0$.
\end{remark}

It will be useful to prepare estimates for the gauge transform
in Sobolev spaces.

\begin{lemma}[Boundedness of the gauge transform]\label{lem:gaugeest}
  Assume $\phi:\mathbb{R}^{n}\to \mathbb{R}$ satisfies
  $\partial \phi\in L^{n,\infty}$.
  Then we have 
  \begin{equation*}
    \|e^{i \phi}v\|_{\dot H^{s}_{p}}\simeq
    \|v\|_{\dot H^{s}_{p}},
    \qquad
    \|e^{i \phi}v\|_{\dot H^{-s}_{p'}}\simeq
    \|v\|_{\dot H^{-s}_{p'}}
  \end{equation*}
  for all
  $s\in[0,1]$ and $1<p<\frac ns$
  i.e. $\frac{n}{n-s}<p'<\infty$.
\end{lemma}

\begin{proof}
  Let $Tv:=e^{i \phi(x)}v$ be the multiplication operator.
  $T$ is an isometry of $L^{p}$ into itself
  for all $p\in[1,\infty]$. Moreover
  \begin{equation*}
    \|Tv\|_{\dot H^{1}_{p}}
    \simeq
    \|\partial Tv\|_{\dot H^{1}_{p}}
    \le
    \|v\partial \phi\|_{L^{p}}+\|\partial v\|_{L^{p}}
    \lesssim
    \|\partial \phi\|_{L^{n,\infty}}
    \|v\|_{\frac{np}{n-p},p}
    +\|\partial v\|_{L^{p}}
  \end{equation*}
  and by Sobolev embedding in Lorentz spaces
  \begin{equation*}
    \|v\|_{\frac{np}{n-p},p} \lesssim\|v\|_{\dot H^{1}_{p}}
  \end{equation*}
  valid for $1<p<n$, we deduce that
  $T$ is a bounded operator on $\dot H^{1}_{p}$ provided $1<p<n$.
  Thus by complex interpolation we obtain that
  $T$ is bounded on $\dot H^{s}_{p}$ provided
  $1<p<n/s$, and since $T^{-1}$ i.e. multiplication by
  $e^{-i \phi}$ enjoys the same property,
  the first claim is proved. 
  The second claim follows by duality.
\end{proof}

We can now give a version of Theorem \ref{the:fullresest}
improved with the use of the gauge transform,
as mentioned in Remark \ref{rem:delta0}:

\begin{corollary}[]\label{cor:equivass}
  Let $n\ge3$.
  Assume the operator $L$ defined in \eqref{eq:opL} is selfadjoint
  and non negative on $L^{2}(\mathbb{R}^{n})$, with domain
  $H^{2}(\mathbb{R}^{n})$.
  Assume $V:\mathbb{R}^{n}\to \mathbb{R}$,
  $A:\mathbb{R}^{n}\to \mathbb{R}^{n}$ 
  and $\phi:\mathbb{R}^{n}\to \mathbb{R}$ satisfy
  for some $\delta>0$
  \begin{equation}\label{eq:assVAequiv`'}
    |x|^{2}\bra{x}^{\delta}(V-i \partial \cdot A-i \Delta \phi)
    \quad\text{and}\quad 
    |x|\bra{x}^{\delta}(|A|+|\partial \phi|+|B|)
    \quad\text{belong to}\quad 
    \ell^{1}L^{\infty}.
  \end{equation}
  Moreover, assume 0 is not a resonance for $L$,
  in the sense of Definition \ref{def:reson}.

  Then estimates \eqref{eq:resestim2} and \eqref{eq:onehalfres}
  are valid. 
\end{corollary}

\begin{proof}
  We apply the gauge transformation \eqref{eq:gaugefor}.
  By assumption the new potential $\widetilde{A}=A+\partial \phi$
  satisfies \eqref{eq:assVA}, while the magnetic field $B$
  does not change, since
  $\partial_{j}(\partial_{k}\phi)-\partial_{k}(\partial_{j} \phi)
    =0$. 
  Thus we are in position to apply Theorem \ref{the:fullresest}
  and we obtain that the resolvent operator
  $\widetilde{R}(z)=(\widetilde{L}-z)^{-1}$,
  where $\widetilde{L}=-\Delta_{\widetilde{A}}+V$,
  satisfies estimate \eqref{eq:resestim2}. Since
  \begin{equation*}
    \widetilde{R}(z)=e^{i \phi}R(z)e^{-i \phi},
  \end{equation*}
  this gives immediately the uniform boundedness of
  $R(z):\dot Y^{*}\to \dot X$ and
  $|z|^{1/2}R(z):\dot Y^{*}\to \dot Y$.
  For the derivative term, we have
  \begin{equation*}
    \|\partial R(z)f\|_{\dot Y}
    \le
    \|\partial \widetilde{R}(z)e^{i \phi}f\|_{\dot Y}
    +
    \|(\partial \phi) \widetilde{R}(z)e^{i \phi}f\|_{\dot Y}.
  \end{equation*}
  The first term is bounded by $\dot Y^{*}$ thanks to the
  estimate for $\widetilde{R}(z)$. For the second term,
  we note that the assumptions on $\phi$ imply
  $|\partial \phi|\lesssim|x|^{-1}$
  and hence we can write
  \begin{equation*}
    \|(\partial \phi) \widetilde{R}(z)e^{i \phi}f\|_{\dot Y}
    \lesssim
    \||x|^{-1} \widetilde{R}(z)e^{i \phi}f\|_{\dot Y}
    \le
    \|\widetilde{R}(z)e^{i \phi}f\|_{\dot X}
    \lesssim
    \|f\|_{\dot Y^{*}}
  \end{equation*}
  and the proof of \eqref{eq:resestim2} for $R(z)$ is concluded.
  The second estimate \eqref{eq:onehalfres} is proved
  by duality and interpolation as in the proof of
  Corollary \ref{cor:onehalfres}.
\end{proof}

\section{Smoothing estimates}\label{sec:smoo_esti}

Using the Kato smoothing theory, the resolvent estimates
of the previos section can be convterted into estimates
for the time--dependent Schr\"{o}dinger flow with little effort.
The theory was initiated in \cite{Kato65-b} and took
the final fomr in \cite{KatoYajima89-a}
(see also \cite{ReedSimon78-a}, \cite{Mochizuki10-a}); it was
further expanded in \cite{DAncona15-a} to include in
the general theory also the wave and Klein--Gordon flows.
Here we follow the formulation\footnote[1]{We take the chance to correct a couple of typos in
\cite{DAncona15-a}, in the definition of smoothing operators and
in the statement of Theorem \ref{the:2abs}.}
of \cite{DAncona15-a}.

Let $\mathcal{H}$, $\mathcal{H}_{1}$ be two Hilbert spaces
and $H$ a selfadjoint operator in $\mathcal{H}$.
Denote with $R(z)=(H-z)^{-1}$ the resolvent operator of $H$,
and with $\Im R(z)=2^{-1}(R(z)-R(z)^{*})$ its imaginary part.

\begin{definition}[Smoothing operator]\label{def:smooop}
  A closed operator $A$ from $\mathcal{H}$ to $\mathcal{H}_{1}$
  with dense domain $D(A)$ is called:
  \\
  (i) $H$-\emph{smooth}, with constant $a$,
  if $\exists\epsilon_{0}$ such that for every 
  $\epsilon,\lambda\in \mathbb{R}$ with 
  $0<|\epsilon|< \epsilon_{0}$ the following
  uniform bound holds:
  \begin{equation}\label{eq:smoo}
    |(\Im R(\lambda+i \epsilon)A^{*}v,A^{*}v)_{\mathcal{H}}|
    \le
    a\|v\|_{\mathcal{H}_{1}}^{2},
    \qquad
    v\in D(A^{*});
  \end{equation}
  (ii) $H$-\emph{supersmooth}, with constant $a$, if 
  in place of \eqref{eq:smoo} one has
  \begin{equation}\label{eq:ssmoo}
    |(R(\lambda+i \epsilon)A^{*}v,A^{*}v)_{\mathcal{H}}|
    \le
    a\|v\|_{\mathcal{H}_{1}}^{2},
    \qquad
    v\in D(A^{*}).
  \end{equation}
 \end{definition}

The following result is proved in 
Lemma 3.6 and Theorem 5.1 of \cite{Kato65-b} 
(see also Theorem XIII.25 in \cite{ReedSimon78-a}).
Here $L^{2} \mathcal{H}$ denotes the space of $L^{2}$
functions on $\mathbb{R}$ with values in $\mathcal{H}$:

\begin{theorem}\label{the:1abs}
  Let $A:\mathcal{H}\to \mathcal{H}_{1}$ be a closed operator
  with dense domain $D(A)$. 
  Then $A$ is $H$-smooth
  with constant $a$ if and only if, for any 
  $v\in \mathcal{H}$, one has
  $e^{-itH}v\in D(A)$ for almost every $t$
  and the following estimate holds:
  \begin{equation}\label{eq:smooest}
    \|Ae^{-itH}v\|_{L^{2}\mathcal{H}_{1}}
    \le
    2a^{\frac12}\|v\|_{\mathcal{H}}.
  \end{equation}
\end{theorem}

Thus $H$-smoothness is equivalent to the smoothing
estimate \eqref{eq:smooest} for the homogeneous
flow $e^{-itH}$. In a similar way,
$H$-supersmoothness is equivalent to a
\emph{nonhomogeneous} estimate:

\begin{theorem}[\cite{DAncona15-a}]\label{the:2abs}
  Let $A:\mathcal{H}\to \mathcal{H}_{1}$ be a closed operator
  with dense domain $D(A)$. 
  Assume $A$ is $H$-supersmooth with constant $a$.
  Then
  $e^{-itH}v\in D(A)$ for almost any $t\in \mathbb{R}$
  and any $v\in \mathcal{H}$; moreover,
  for any step function 
  $h(t):\mathbb{R}\to D(A^{*})$,
  $Ae^{-i(t-s)H}A^{*}h(s)$ is Bochner integrable in $s$ 
  over $[0,t]$ (or $[t,0]$)
  and satisfies, for all $\epsilon\in(-\epsilon_{0},\epsilon_{0})$,
  the estimate
  \begin{equation}\label{eq:ssmooest}
    \textstyle
    \|e^{-|\epsilon| t}\int_{0}^{t}Ae^{-i(t-s)H}A^{*}h(s)ds\|
    _{L^{2}\mathcal{H}_{1}}
    \le 2a
    \|e^{-|\epsilon| t}h(t\|_{L^{2}\mathcal{H}_{1}}.
  \end{equation}
  Conversely, if \eqref{eq:ssmooest} holds, then $A$
  is $H$-supersmooth with constant $2a$.
\end{theorem}

The extension to the wave
and Klein--Gordon groups
is the following:

\begin{theorem}[\cite{DAncona15-a}]\label{the:waveP}
  Let $\nu\in \mathbb{R}$ with $H+\nu\ge0$ and let $P$
  be the orthogonal projection onto $\ker(H+\nu)^{\perp}$.
  Assume $A$ and $A(H+\nu)^{-\frac14}P$ are closed operators with
  dense domain from $\mathcal{H}$ to $\mathcal{H}_{1}$.
  \\
  (i) If $A$ is $H$-smooth with constant $a$, then
  $A(H+\nu)^{-\frac14}P$ is $\sqrt{H+\nu}$-smooth
  with constant $C=(\pi+3)a$.
  In particular, we have the estimate
  \begin{equation}\label{eq:sqsmoo2}
    \|Ae^{-it \sqrt{H+\nu}}v\|_{L^{2}\mathcal{H}_{1}}
    \le
    2C^{\frac12}\|(H+\nu)^{\frac14}v\|_{\mathcal{H}},
    \qquad
    \forall v\in D((H+\nu)^{\frac14}).
  \end{equation}
  (ii) If $A$ is $H$-supersmooth with constant $a$, then
  $A(H+\nu)^{-\frac14}P$ is $\sqrt{H+\nu}$-supersmooth
  with constant $C=(\pi+3)a$.
  In particular, we have the estimate
  \begin{equation}\label{eq:sqssmoo2}
    \textstyle
    \|\int_{0}^{t}Ae^{-i(t-s)\sqrt{H+\nu}}
      (H+\nu)^{-\frac12}PA^{*}h(s)ds\|
      _{L^{2}\mathcal{H}_{1}}
    \le
    C\|h\| _{L^{2}\mathcal{H}_{1}}
  \end{equation}
  for any step function 
  $h:\mathbb{R}\to D((H+\nu)^{-\frac14}PA^{*})$.
\end{theorem}

We can now recast the resolvent estimates of 
Corollary \ref{cor:onehalfres} in the framework of the
Kato--Yajima theory:

\begin{corollary}[]\label{cor:Lsmoothing}
  Let $\rho(x)$ be an arbitrary function in $\ell^{2}L^{\infty}$.
  Assume the operator $L$ defined by \eqref{eq:opL} satisfies the
  assumptions of 
  either Theorem \ref{the:fullresest}
  or Corollary \ref{cor:equivass}. Then the operators
  \begin{equation*}
    \rho(x)|x|^{-1}
    \qquad\text{and}\qquad 
    \rho(x)|x|^{-1/2}|D|^{1/2}
  \end{equation*}
  are $L$--supersmooth (and hence $L$--smooth), 
  with a constant of the form
  $C\|\rho\|_{\ell^{2}L^{\infty}}^{2}$.
  If in addition $|x|\partial \rho\in \ell^{2}L^{\infty}$, 
  then the operator
  \begin{equation*}
    |D|^{1/2}\rho|x|^{-1/2}
  \end{equation*}
  is also $L$--supersmooth, with a constant
  $C\||\rho|+|x||\partial \rho|\|_{\ell^{2}L^{\infty}}^{2}$.
\end{corollary}

\begin{proof}
  Note that
  \begin{equation*}
    \|\rho|x|^{-1}v\|_{L^{2}}
    \le
    \|\rho\|_{\ell^{2}L^{\infty}}
    \||x|^{-1}v\|_{\ell^{\infty}L^{2}}
    =
    \|\rho\|_{\ell^{2}L^{\infty}}
    \||x|^{-1/2}v\|_{\dot Y}.
  \end{equation*}
  Thus we have, by \eqref{eq:onehalfres},
  \begin{equation*}
  \begin{split}
    \|\rho|x|^{-1}R(z)|x|^{-1/2}f\|_{L^{2}}
    \lesssim
    \|\rho\|_{\ell^{2}L^{\infty}}
    \|f\|_{\dot Y^{*}}
    =&
    \|\rho\|_{\ell^{2}L^{\infty}}
    \|\rho \rho^{-1}|x|^{1/2}f\|_{\ell^{1}L^{2}}
    \\
    \le&
    \|\rho\|_{\ell^{2}L^{\infty}}^{2}
    \|\rho^{-1}|x|^{1/2}f\|_{L^{2}}
    \end{split}
  \end{equation*}
  which can be written
  \begin{equation*}
    \|\rho|x|^{-1}R(z)|x|^{-1}\rho g\|_{L^{2}}
    \lesssim
    \|\rho\|_{\ell^{2}L^{\infty}}^{2}
    \|g\|_{L^{2}}.
  \end{equation*}
  The proof for $\rho(x)|x|^{-1/2}|D|^{1/2}$ is similar.
  For the last operator, we first note that
  for any $y\in \mathbb{R}$
  \begin{equation*}
  \begin{split}
    \||D|^{1+iy}\rho|x|^{-1/2}R(z)
      &
      f\|_{L^{2}}
    \simeq
    \|\partial (\rho|x|^{-1/2}R(z)f)\|_{L^{2}}
    \\
    \le&
    \|\rho|x|^{-1/2}\partial R(z)f\|_{L^{2}}
    +
    \|(\partial (\rho|x|^{-1/2}))R(z)f)\|_{L^{2}}.
    \end{split}
  \end{equation*}
  The first term at the right can be bounded by the
  $\dot Y$ norm and hence by $\|f\|_{\dot Y^{*}}$
  thanks to \eqref{eq:resestim2}.
  The second term is bounded by
  \begin{equation*}
    \textstyle
    \|(|x|\partial \rho-\frac{\widehat{x}}{2}\rho)
    |x|^{-3/2}R(z)f\|_{L^{2}}
    \lesssim
    \||x||\partial \rho|+\rho\|
      _{\ell^{2}L^{\infty}}
    \||x|^{-1}R(z)f\|_{\dot Y}
  \end{equation*}
  and hence again by $\|f\|_{\dot Y^{*}}$, using the inequality
  $\||x|^{-1}v\|_{\dot Y}\le\|v\|_{\dot X}$ and again 
  \eqref{eq:resestim2}. In conclusion we have
  \begin{equation*}
    \||D|^{1+iy}\rho|x|^{-1/2}R(z) f\|_{L^{2}}
    \le
    C(\|\rho\|_{\ell^{2}L^{\infty}}+
      \||x|\partial \rho\|_{\ell^{2}L^{\infty}})
    \|f\|_{\dot Y^{*}}
  \end{equation*}
  which implies that the operator
  \begin{equation*}
    |D|^{1+iy}\rho|x|^{-1/2}R(z)|x|^{-1/2}\rho
  \end{equation*}
  is bounded on $L^{2}$, with a constant 
  $C(\|\rho+|x||\partial \rho|\|_{\ell^{2}L^{\infty}})
      \|\rho\|_{\ell^{2}L^{\infty}}$
  where $C$ is independent of $z$ or $y$. By duality the same
  holds for the operatpr
  \begin{equation*}
    \rho|x|^{-1/2}R(z)|x|^{-1/2}\rho|D|^{1+iy}.
  \end{equation*}
  Hence we can apply Stein--Weiss interpolation to the
  analytic family of operators
  \begin{equation*}
    T_{w}=
    |D|^{1-w}\rho|x|^{-1/2}R(z)|x|^{-1/2}\rho|D|^{w}
  \end{equation*}
  with $w$ in the complex strip $0\le\Re w\le1$, as in the
  proof of Corollary \ref{cor:onehalfres}.
  Taking $w=1/2$ we conclude the proof.
\end{proof}

Then applying the abstract theory we obtain immediately:

\begin{corollary}[Smoothing for Schr\"{o}dinger]\label{cor:smoosch}
  Under the assumptions of either Theorem \ref{the:fullresest}
  or Corollary \ref{cor:equivass}, we have
  for any $\rho>0$ in $\ell^{2}L^{\infty}$
  \begin{equation}\label{eq:smooschh}
    \|\rho|x|^{-1}e^{itL}f\|_{L^{2}_{t}L^{2}}
    +
    \|\rho|x|^{-1/2}|D|^{1/2}e^{itL}f\|_{L^{2}_{t}L^{2}}
    \le
    C\|\rho\|_{\ell^{2}L^{\infty}}\|f\|_{L^{2}}
  \end{equation}
  \begin{equation}\label{eq:smooschnh1}
    \textstyle
    \left\|
      \rho|x|^{-1}\int_{0}^{t}e^{i(t-t')L}F(t')dt'\right\|
        _{L^{2}_{t}L^{2}}
      \le
      C \|\rho\|_{\ell^{2}L^{\infty}}^{2}
      \|\rho^{-1}|x|F\|_{L^{2}_{t}L^{2}}
  \end{equation}
  \begin{equation}\label{eq:smooschnh2}
    \textstyle
    \left\|
      \rho|x|^{-1/2}|D|^{1/2}\int_{0}^{t}e^{i(t-t')L}F(t')dt'\right\|
        _{L^{2}_{t}L^{2}}
      \le
      C \|\rho\|_{\ell^{2}L^{\infty}}^{2}
      \|\rho^{-1}|x|^{1/2}|D|^{-1/2}F\|_{L^{2}_{t}L^{2}}
  \end{equation}
  with a constant $C$ independent of $\rho$.

  If in addition $|x|\partial \rho\in \ell^{2}L^{\infty}$,
  then the previous estimates hold with the operator
  $\rho|x|^{-1/2}|D|^{1/2}$ replaced by
  $|D|^{1/2}\rho|x|^{-1/2}$ in \eqref{eq:smooschh}
  and \eqref{eq:smooschnh2}, the operator
  $\rho^{-1}|x|^{1/2}|D|^{-1/2}$ replaced by
  $|D|^{-1/2}\rho|x|^{1/2}$ in the last term in 
  \eqref{eq:smooschnh2}, and the norm
  $\|\rho\|_{\ell^{2}L^{\infty}}$ replaced by
  $\|\rho+|x||\partial \rho|\|_{\ell^{2}L^{\infty}}$.
  In particular we have
  \begin{equation}\label{eq:smoosch3}
    \||D|^{1/2}\rho|x|^{-1/2}e^{itL}f\|_{L^{2}_{t}L^{2}}
    \lesssim\|f\|_{L^{2}}.
  \end{equation}
\end{corollary}

Before stating the corresponding estimates for the wave
equation we prove a simple bound for the powers of the
operator $L$:

\begin{lemma}[]\label{lem:sobolevL}
  Assume $-L=(\partial+iA)^{2}-V$ is selfadjoint and non negative
  in $L^{2}(\mathbb{R}^{n})$, $n\ge3$, and that 
  $|V|+|A|^{2}\lesssim|x|^{-2}$. Then for all $0\le s\le 1$
  we have
  \begin{equation}\label{eq:boundsobL}
    \|L^{s/2}v\|_{L^{2}}\lesssim
    \||D|^{s}v\|_{L^{2}},
    \qquad
    \||D|^{-s}v\|_{L^{2}}\lesssim
    \|L^{-s/2}v\|_{L^{2}}.
  \end{equation}
\end{lemma}

\begin{proof}
  The second estimate is equivalent to the first one by duality.
  It is sufficient to prove the first estimate for $s=1$ and then
  interpolate with the trivial case $s=0$.
  When $s=1$ we have
  \begin{equation*}
    \|L^{1/2}v\|_{L^{2}}
    =
    (Lv,v)
    \lesssim
    \|\partial v\|^{2}_{L^{2}}+\||x|^{-1}v\|_{L^{2}}^{2}
    \lesssim\||D|v\|_{L^{2}}^{2}
  \end{equation*}
  by Hardy's inequality.
\end{proof}

For the wave flow $e^{it \sqrt{L}}$ and the Klein--Gordon
flow $e^{it \sqrt{L+1}}$ we have then:

\begin{corollary}[Smoothing for Wave--K--G]\label{cor:smooWE}
  Under the assumptions of either Theorem \ref{the:fullresest}
  or Corollary \ref{cor:equivass}, we have
  for any $\rho>0$ in $\ell^{2}L^{\infty}$
  \begin{equation}\label{eq:smooWEh}
    \|\rho|x|^{-1}e^{it\sqrt{L}}f\|_{L^{2}_{t}L^{2}}
    +
    \|\rho|x|^{-1/2}|D|^{1/2}e^{it \sqrt{L}}f\|_{L^{2}_{t}L^{2}}
    \le
    C\|\rho\|_{\ell^{2}L^{\infty}}\|L^{1/4}f\|_{L^{2}}
  \end{equation}
  where the last term can be estimated by
  $C'\|\rho\|_{\ell^{2}L^{\infty}}
    \|f\|_{\dot H^{1/2}}$,
  \begin{equation}\label{eq:smooWEnh1}
    \textstyle
    \left\|
      \rho|x|^{-1}\int_{0}^{t}
        \frac{e^{i(t-t')\sqrt{L}}}{\sqrt{L}}
        F(t')dt'\right\|
        _{L^{2}_{t}L^{2}}
      \le
      C \|\rho\|_{\ell^{2}L^{\infty}}^{2}
      \|\rho^{-1}|x|F\|_{L^{2}_{t}L^{2}}
  \end{equation}
  \begin{equation}\label{eq:smooWEnh2}
    \textstyle
    \left\|
      \rho|x|^{-1/2}|D|^{1/2}\int_{0}^{t}
        \frac{e^{i(t-t')\sqrt{L}}}{\sqrt{L}}
        F(t')dt'\right\|
        _{L^{2}_{t}L^{2}}
      \le
      C \|\rho\|_{\ell^{2}L^{\infty}}^{2}
      \|\rho^{-1}|x|^{1/2}|D|^{-1/2}F\|_{L^{2}_{t}L^{2}}
  \end{equation}
  with constants $C,C'$ independent of $\rho$.

  If in addition $|x|\partial \rho\in \ell^{2}L^{\infty}$,
  then the previous estimates hold with the operator
  $\rho|x|^{-1/2}|D|^{1/2}$ replaced by
  $|D|^{1/2}\rho|x|^{-1/2}$ in 
  \eqref{eq:smooWEh} and \eqref{eq:smooWEnh2}, 
  $\rho^{-1}|x|^{1/2}|D|^{-1/2}$ replaced by
  $|D|^{-1/2}\rho|x|^{1/2}$ in the last term in 
  \eqref{eq:smooWEnh2}, and the norm
  $\|\rho\|_{\ell^{2}L^{\infty}}$ replaced by
  $\|\rho+|x||\partial \rho|\|_{\ell^{2}L^{\infty}}$.
  In particular we have
  \begin{equation}\label{eq:smooWE3}
    \||D|^{1/2}\rho|x|^{-1/2}e^{it \sqrt{L}}f\|_{L^{2}_{t}L^{2}}
    \lesssim\|L^{1/4}f\|_{L^{2}}.
  \end{equation}

  The same estimates hold if we replace $L$ by $L+1$
  everywhere; in this case the last term in \eqref{eq:smooWEh}
  must be estimated by
  $C'\|\rho\|_{\ell^{2}L^{\infty}}
    \|f\|_{H^{1/2}}$
  (nonhomogeneous norm).
\end{corollary}

\begin{proof}
  Estimate \eqref{eq:smooWEh} follows from 
  Corollary \ref{cor:Lsmoothing}, \eqref{eq:sqsmoo2}
  and \eqref{eq:boundsobL} with $s=1/2$.
  Estimates \eqref{eq:smooWEnh1}, \eqref{eq:smooWEnh2} are a direct
  application of \eqref{eq:sqssmoo2}.
  The other claims are proved in a similar way.
\end{proof}

We note that the previous smoothing estimates can be
put into a scale invariant (but equivalent) form.
Indeed, one has the equivalence
\begin{equation}\label{eq:equivdY}
  \|v\|_{\dot Y}=
  \sup_{\|\rho\|_{\ell^{\infty}L^{2}}=1}
  \|\rho|x|^{-1/2}v\|_{L^{2}}
\end{equation}
which is obtained choosing $\rho=\one{C_{j}}$
with $C_{j}=\{x \colon 2^{j}\le|x|<2^{j+1}\}$
and taking the supremum over $j\in \mathbb{Z}$.
Thus we obtain the follwing result:

\begin{corollary}[]\label{cor:smooswap}
  Under the assumptions of either Theorem \ref{the:fullresest}
  or Corollary \ref{cor:equivass}, we have
  \begin{equation}\label{eq:smooschhsw}
    \||x|^{-1/2}e^{itL}f\|_{\dot YL^{2}_{t}}
    +
    \||D|^{1/2}e^{itL}f\|_{\dot YL^{2}_{t}}
    \le
    C\|f\|_{L^{2}},
  \end{equation}
  \begin{equation}\label{eq:smooWEhsw}
    \||x|^{-1/2}e^{it\sqrt{L}}f\|_{\dot YL^{2}_{t}}
    +
    \||D|^{1/2}e^{it \sqrt{L}}f\|_{\dot YL^{2}_{t}}
    \le
    C\|f\|_{\dot H^{1/2}},
  \end{equation}
  \begin{equation}\label{eq:smooKGhsw}
    \||x|^{-1/2}e^{it\sqrt{L+1}}f\|_{\dot YL^{2}_{t}}
    +
    \||D|^{1/2}e^{it \sqrt{L+1}}f\|_{\dot YL^{2}_{t}}
    \le
    C\|f\|_{H^{1/2}}.
  \end{equation}
\end{corollary}

\section{Strichartz estimates}\label{sec:stri_esti}

We first prove a simple extension 
to Lorentz spaces of the
Muckenhoupt--Wheeden weighted fractional integration
estimate. In the course of the proof of Strichartz
estimates we shall actually need only 
\eqref{eq:muwleb} and \eqref{eq:commL2a},
but we included the next two Lemmas to give a simple
alternative proof of the 
H\"{o}rmander--Plancherel identities in the Appendix
of \cite{ErdoganGoldbergSchlag09-a}, which were crucial
to their result.

\begin{lemma}[Weighted Sobolev embedding]\label{lem:muckwhee}
  For all $1<p\le q<\infty$
  the following inequality holds:
  \begin{equation}\label{eq:muwleb}
    \|v g\|_{L^{q}}\leq C\|v|D|^{\alpha} g\|_{L^{p}},
    \qquad
    \frac{\alpha}{n}=\frac{1}{p}-\frac{1}{q}
  \end{equation}
  for all weights $v\in A_{2-\frac1p}\cap RH_{q}$
  or, equivalently, such that $v^{q}\in A_{1+\frac{q}{p'}}$.
  More generally, for any $r\in[1,\infty]$ and $p,q,\alpha$
  as above we have the inequality in weeighted Lorentz norms
  \begin{equation}\label{eq:muwlor}
    \|v g\|_{L^{q,r}}\leq C\|v |D|^{\alpha} g\|_{L^{p,r}}
  \end{equation}
  provided the weight $v$ satisfies
  $v^{q+\epsilon}\in A_{1+\frac{q}{p'}-\epsilon}$
  for some $\epsilon>0$.
\end{lemma}

\begin{proof}
  Estimate \eqref{eq:muwleb} for $v\in A_{2-\frac1p}\cap RH_{q}$
  is due to Muckenhoupt and Wheeden \cite{MuckenhouptWheeden74-a}
  (see also \cite{AuscherMartell08-a}). The equivalent
  condition on the weight is easy to check, see e.g.
  \cite{JohnsonNeugebauer91-a}.
  To prove the last statement, fix $\delta>0$ sufficiently small
  and write
  \begin{equation*}
    \textstyle
    p_{+}=\frac{p}{1-\delta p},
    \qquad
    p_{-}=\frac{p}{1+\delta p}
    \quad\implies\quad
    \frac{1}{p_{\pm}}=\frac1p\mp \delta
    \qquad
    p_{-}<p<p_{+}
  \end{equation*}
  and similarly for $q_{\pm}$.
  Then \eqref{eq:muwleb} holds for the couples
  $(p_{+},q_{+})$ and $(p_{-},q_{-})$ with $\alpha$ unchanged,
  and hence by real interpolation we get \eqref{eq:muwlor},
  provided the weight $v$ satisfies
  \begin{equation*}
    v^{q_{\pm}}\in A_{1+\frac{q_{\pm}}{(p_{\pm})'}},
  \end{equation*}
  which are implied by 
  $w^{q_{+}}\in A_{1+\frac{q_{-}}{(p_{-})'}}$ thanks to
  the inclusion properties of $A_{p}$ classes.
\end{proof}

\begin{lemma}[]\label{lem:L2bddcomm}
  Let $\sigma\in L^{1}_{loc}(\mathbb{R}^{n})$ be such that
  $\sigma^{2}\in A_{2}$ and
  $|\partial \sigma^{-1}|\lesssim \sigma^{-1}|x|^{-1}$.
  Then the following operator is bounded on $L^{2}$:
  \begin{equation}\label{eq:commL2a}
    |D|^{-1/2}\sigma^{-1}|D|^{1/2}\sigma.
  \end{equation}
  If in addition
  $\sigma^{-n-\epsilon}\in A_{1+\frac{n}{n-2}-\epsilon}$
  for some $\epsilon>0$ then the following operator is also
  bounded on $L^{2}$:
  \begin{equation}\label{eq:commL2}
    |D|^{1/2}\sigma^{-1}|D|^{-1/2}\sigma.
  \end{equation}
\end{lemma}

\begin{proof}
  We prove \eqref{eq:commL2} first.
  Consider the analytic family of operators
  \begin{equation*}
    U_{z}=|D|^{z}\sigma|D|^{-z}\sigma^{-1},
    \qquad
    z\in \mathbb{C},
    \qquad
    0\le\Im z\le1.
  \end{equation*}
  For $z=iy$, $y\in \mathbb{R}$, we have
  \begin{equation*}
    \|U_{iy}f\|_{L^{2}}
    =\|\sigma|D|^{-iy}\sigma^{-1}f\|_{L^{2}}
    \lesssim
    \|\sigma \sigma^{-1}f\|_{L^{2}}=\|f\|_{L^{2}}
  \end{equation*}
  where we used the well--known
  fact that $|D|^{iy}$ is bounded
  in weighted $L^{2}$ if the weight is in $A_{2}$; note also
  that the implicit constant grows at most polynomially
  in $y$ (actually $\lesssim|y|^{n/2}$, see e.g.
  \cite{CacciafestaDAncona12-a}). On the other hand,
  for $z=1+iy$, $y\in \mathbb{R}$ we can write
  \begin{equation*}
    \|U_{1+iy}f\|_{L^{2}}
    \simeq
    \|\partial(\sigma^{-1}|D|^{-1-iy}\sigma f)\|_{L^{2}}
    \le
    \||x|^{-1}\sigma^{-1}|D|^{-1-iy}\sigma f\|_{L^{2}}
    +
    \|\sigma^{-1}\partial|D|^{-1-iy}\sigma f\|_{L^{2}}.
  \end{equation*}
  Since $\sigma^{-2}\in A_{2}$, we have
  \begin{equation*}
    \|\sigma^{-1}\partial|D|^{-1-iy}\sigma f\|_{L^{2}}
    \simeq
    \|\sigma^{-1}|D||D|^{-1-iy}\sigma f\|_{L^{2}}
    \simeq
    \|f\|_{L^{2}}.
  \end{equation*}
  On the other hand,
  \begin{equation*}
    \||x|^{-1}\sigma^{-1}|D|^{-1-iy}\sigma f\|_{L^{2}}
    \lesssim
    \|\sigma^{-1}|D|^{-1-iy}\sigma f\|_{L^{\frac{2n}{n-2},2}}
    \lesssim
    \|\sigma^{-1}|D|^{-iy}f\|_{L^{2}}
    \simeq\|f\|_{L^{2}}
  \end{equation*}
  using \eqref{eq:muwlor} with the choice
  $q=\frac{2n}{n-2}$, $p=r=2$ and $\alpha=1$, provided
  $\sigma^{-n-\epsilon}\in A_{1+\frac{n}{n-2}-\epsilon}$.
  By Stein--Weiss interpolation we obtain that
  $U_{z}$ is bounded in $L^{2}$ for all values in the strip,
  and this gives the claim taking $z=1/2$.

  Consider now the operator \eqref{eq:commL2a},
  or equivalently its adjoint,
  which we denote also by
  \begin{equation*}
    T=\sigma|D|^{1/2}\sigma^{-1}|D|^{-1/2}.
  \end{equation*}
  To prove that $T$ is bounded on $L^{2}$, we use the analytic
  family of operators
  \begin{equation*}
    U_{w}=\sigma|D|^{1/2}\sigma^{-1}
  \end{equation*}
  for $w$ in the complex strip $0\le \Re w\le1$.
  The operator
  \begin{equation*}
    U_{iy}=\sigma|D|^{iy}\sigma^{-1}
  \end{equation*}
  for $y\in \mathbb{R}$ is bounded on $L^{2}$ with a
  growth at most polynomial in $|y|$ as $|y|\to \infty$,
  provided $\sigma\in A_{2}$.
  On the otner hand, for $w=1+iy$ we can write
  \begin{equation*}
    \|U_{1+iy}f\|_{L^{2}}
    =
    \|\sigma|D|^{iy}|D|\sigma^{-1}f\|_{L^{2}}
    \le
    \|\sigma \partial (\sigma^{-1}f)\|_{L^{2}}
  \end{equation*}
  and if we have the property 
  \begin{equation}\label{eq:assw4}
    |\partial \sigma^{-1}|\lesssim \sigma^{-1}|x|^{-1}
  \end{equation}
  we can continue
  \begin{equation*}
    \|U_{1+iy}f\|_{L^{2}}
    \lesssim
    \||x|^{-1}f\|_{L^{2}}+\|\partial f\|_{L^{2}}
    \lesssim
    \|\partial f\|_{L^{2}}
    \simeq
    \|f\|_{\dot H^{1}}
  \end{equation*}
  by Hardy's inequality; again, the implicit constant grows
  at most polynomially in $y$. By Stein--Weiss complex
  interpolation we deduce the estimate
  \begin{equation*}
    \|U_{1/2}f\|_{L^{2}}
    =
    \|\sigma|D|^{1/2}\sigma^{-1}f\|_{L^{2}}
    \lesssim\|f\|_{\dot H^{1/2}}
  \end{equation*}
  which implies
  \begin{equation*}
    \|Tf\|_{L^{2}}
    =
    \|\sigma|D|^{1/2}\sigma^{-1}|D|^{-1/2}f\|_{L^{2}}
    \lesssim\|f\|_{L^{2}}.
    \qedhere
  \end{equation*}
\end{proof}

To handle endpoint Strichartz estimates
we resort to a
mixed endpoint Strichartz--smoothing estimate
for the free flow, due to
Ionescu and Kenig \cite{IonescuKenig05-a}:
\begin{equation}\label{eq:iok1}
  \textstyle
  \|\int_{0}^{t}e^{-i(t-t')\Delta}F(t')dt'\|_{L^{2}_{t}L^{2^{*}}}
  \lesssim
  \min_{j=1,\dots,n}
  \||D_{j}|^{-1/2}F\|_{L^{2}_{t}L_{x_{j}}^{1}L^{2}_{x'_{j}}},
  \qquad
  2^{*}=\frac{2n}{n-2}
\end{equation}
where the norm at the right are $L^{1}$ with respect to one of
the coordinates and $L^{2}$ with respect to the remaining 
coordinates. By an easy modification of the argument in
\cite{IonescuKenig05-a}, as observed by Mizutani \cite{Mizutani16-b}
one can refine \eqref{eq:iok1} to an estimate in the Lorentz
norm $L^{2^{*},2}$
\begin{equation}\label{eq:iok2}
  \textstyle
  \|\int_{0}^{t}e^{-i(t-t')\Delta}F(t')dt'\|_{L^{2}_{t}L^{2^{*},2}}
  \lesssim
  \min_{j=1,\dots,n}
  \||D_{j}|^{-1/2}F\|_{L^{2}_{t}L_{x_{j}}^{1}L^{2}_{x'_{j}}};
\end{equation}
moreover, if $w>0$ is such that $w^{2}\in A_{2}(\mathbb{R}^{n})$
and there exists $C$ such that
$\int w^{-2}dx_{j}<C$ for $j=1,\dots,n$
(uniformly in the remaining variables),
then we have
\begin{equation*}
  \textstyle
  \sum_{j=1}^{n}
  \||D_{j}|^{-1/2}F\|_{L^{2}_{t}L_{x_{j}}^{1}L^{2}_{x'_{j}}}
  \lesssim
  \sum_{j=1}^{n}
  \|w|D_{j}|^{-1/2}F\|_{L^{2}_{t}L^{2}}
  \simeq
  \|w|D|^{-1/2}F\|_{L^{2}_{t}L^{2}}
\end{equation*}
by the usual weighted estimates for singular integrals.
Thus \eqref{eq:iok1} implies the estimate
\begin{equation}\label{eq:iok}
  \textstyle
  \|\int_{0}^{t}e^{-i(t-t')\Delta}F(t')dt'\|_{L^{2}_{t}L^{2^{*},2}}
  \lesssim
  \|w|D|^{-1/2}F\|_{L^{2}_{t}L^{2}}
\end{equation}
for any weight $w$ as above.

\begin{theorem}[]\label{the:endpointsch}
  Let $n\ge3$.
  Assume the operator $L$ defined in \eqref{eq:opL} is selfadjoint
  and non negative in $L^{2}(\mathbb{R}^{n})$, with domain
  $H^{2}(\mathbb{R}^{n})$.
  Assume $V:\mathbb{R}^{n}\to \mathbb{R}$
  and $A:\mathbb{R}^{n}\to \mathbb{R}^{n}$ satisfy
  for some $\delta>0$ and $\mu>1$
  \begin{equation}\label{eq:assVAendp}
    \bra{\log|x|}^{\mu}\bra{x}^{\delta}
    |x|^{2}(V-i \partial \cdot A)\in L^{\infty},
    \qquad
    \bra{\log|x|}^{\mu}\bra{x}^{\delta}|x| \widehat{B}\in L^{\infty}
  \end{equation}
  and
  \begin{equation}\label{eq:assAendp}
    \bra{\log|x|}^{\mu}\bra{x}^{\delta}|x|A
    \in L^{\infty}\cap \dot H^{1/2}_{2n}.
  \end{equation}
  Moreover, assume 0 is not a resonance for $L$,
  in the sense of Definition \ref{def:reson}.

  Then the Schr\"{o}dinger flow $e^{itL}$ satisfies
  the endpoint Strichartz estimate
  \begin{equation}\label{eq:endpstrsch}
    \|e^{itL}f\|_{L^{2}_{t}L^{\frac{2n}{n-2},2}}
    \lesssim
    \|f\|_{L^{2}}
  \end{equation}
  and the corresponding estimates in $L^{p}_{t}L^{q}$
  for all Schr\"{o}dinger admissible $(p,q)$;
  and the nonhomogeneous estimates
  \begin{equation}\label{eq:endostrschnh}
    \textstyle
    \|\int_{0}^{t}e^{i(t-t')L}F(t')dt'\|_{L^{p}_{t}L^{q}}
    \lesssim
    \|F\|_{L^{\widetilde{p}'}L^{\widetilde{q}'}}
  \end{equation}
  for all Schr\"{o}dinger admissible couples 
  $(p,q)$ and $(\widetilde{p},\widetilde{q})$ such that
  $p<\widetilde{p}'$.
\end{theorem}

\begin{proof}
  Since the assumptions of Corollary \ref{cor:equivass}
  are satisfied, the smoothing estimates
  \eqref{eq:smooschh} and \eqref{eq:smoosch3}
  are valid.
  The flow $u=e^{itL}f$ is the solution of
  \begin{equation*}
    u(0)=f,
    \qquad
    i\partial_{t}u+Lu
    \equiv
    i\partial_{t}u-\Delta_{A}u-Vu=0.
  \end{equation*}
  By Duhamel's formula we can represent $u$ in the form
  \begin{equation}\label{eq:expansionsch}
    \textstyle
    u=e^{-it \Delta}f
    -i\int_{0}^{t}e^{-i(t-t')\Delta}
    \Bigl[2i\partial \cdot(A  u)
    +(V-i \partial \cdot A-|A |^{2})u\Bigr]dt'.
  \end{equation}
  We compute the $L^{2}_{t}L^{2^{*},2}$ norm of $u$.
  To the first term in \eqref{eq:expansionsch} we apply
  \eqref{eq:striendphom}.
  For the remaining terms we use \eqref{eq:iok} and we are
  led to estimate
  \begin{equation*}
    \textstyle
    I=\|\sigma|D|^{-1/2}\widetilde{V}u\| _{L^{2}_{t}L^{2}},
    \qquad
    \widetilde{V}:=V-i \partial \cdot A-|A |^{2},
  \end{equation*}
  \begin{equation*}
    \textstyle
    II=
    \|\sigma|D|^{-1/2}
    \partial \cdot (A  u)\|
    _{L^{2}_{t}L^{2}}
    \simeq
    \|\sigma|D|^{1/2} (A  u)\|
    _{L^{2}_{t}L^{2}}
  \end{equation*}
  where $\sigma$ is any weight on $\mathbb{R}^{n}$ such that
  \begin{equation}\label{eq:assw}
    \textstyle
    \sigma^{2}\in A_{2},
    \qquad
    \int \sigma^{-2}dx_{j}<\infty
    \qquad
    j=1,\dots,n.
  \end{equation}

  The quantity $I$ can be estimated via the weighted
  Sobolev embeddings \eqref{eq:muwleb}:
  with the choices $\alpha=1/2$, 
  $p=\frac{2n}{n+1}$ and $q=r=2$ we obtain
  \begin{equation}\label{eq:muw2}
    \|\sigma u\|_{L^{2}}
    \lesssim
    \|\sigma|D|^{1/2}u\|_{L^{\frac{2n}{n+1},2}}
  \end{equation}
  provided 
  \begin{equation}\label{eq:assw2}
    \sigma^{2+\epsilon} \in A_{2-\frac{1}{n}-\epsilon}
  \end{equation}
  for some $\epsilon>0$ small.
  Then we have, by H\"{o}lder inequaity,
  \begin{equation*}
    I \lesssim
    \|\sigma\widetilde{V}u\|
    _{L^{2}_{t}L^{\frac{2n}{n+1},2}}
    \le
    \|\sigma\rho^{-1}|x|\widetilde{V}\|
    _{L^{2n,\infty}}
    \|\rho|x|^{-1}u\|_{L^{2}_{t}L^{2}}
    \lesssim
    \|\sigma\rho^{-1}|x|\widetilde{V}\|_{L^{2n,\infty}}
    \|f\|_{L^{2}};
  \end{equation*}
  in the last step we used the smoothing estimate
  \eqref{eq:smooschh}. 

  To estimate the quantity $II$, we first commute
  the multiplication operator 
  $\sigma$ with $|D|^{1/2}$. This is possible since the
  operator
  \begin{equation*}
    T=\sigma|D|^{1/2}\sigma^{-1}|D|^{-1/2}
  \end{equation*}
  is bounded in $L^{2}$ by
  Lemma \ref{lem:L2bddcomm}, provided
  \begin{equation}\label{eq:assw3}
    |\partial \sigma^{-1}|\lesssim \sigma^{-1}|x|^{-1}.
  \end{equation}
  Thus we have,
  for any $\rho\in \ell^{2}L^{\infty}$,
  \begin{equation*}
    II \lesssim
    \||D|^{1/2}\sigma(A u)\|
    _{L^{2}_{t}L^{2}}
    =
    \||D|^{1/2}(\sigma \rho^{-1}|x|^{1/2} A)
    (\rho|x|^{-1/2} u)\|_{L^{2}_{t}L^{2}}
  \end{equation*}
  and using the Kato--Ponce inequality
  \begin{equation*}
    \lesssim
    \|\sigma \rho^{-1}|x|^{1/2} A\|_{L^{\infty}}
    \||D|^{1/2}\rho|x|^{-1/2} u\|_{L^{2}_{t}L^{2}}
    +
    \||D|^{1/2}\sigma \rho^{-1}|x|^{1/2} A\|_{L^{2n}}
    \|\rho|x|^{-1/2} u\|_{L^{2}_{t}L^{\frac{2n}{n-1}}}.
  \end{equation*}
  We use Sobolev embedding for the last term, and 
  then smoothing estimate
  \eqref{eq:smoosch3}, and we arrive at
  \begin{equation*}
    II \lesssim
    \|\sigma \rho^{-1}|x|^{1/2} A\|
    _{L^{\infty}\cap \dot H^{1/2}_{2n}}
    \|f\|_{L^{2}}.
  \end{equation*}

  Summing up, we have proved
  \begin{equation*}
    \|e^{itL}f\|_{L^{p}L^{q}}
    \lesssim
    (1+
    \|\sigma \rho^{-1}|x|^{1/2} A\|
    _{L^{\infty}\cap \dot H^{1/2}_{2n}}
    +
    \|\sigma\rho^{-1}|x|\widetilde{V}\|_{L^{2n,\infty}}
    )
    \cdot
    \|f\|_{L^{2}}.
  \end{equation*}
  If we now add the condition
  \begin{equation}\label{eq:assw5}
    \rho \lesssim \sigma|x|^{-1/2}
  \end{equation}
  then we can write
  \begin{equation*}
    \|\sigma \rho^{-1}|x|A^{2}\|_{L^{2n,\infty}}
    \lesssim\|\sigma \rho^{-1}|x|^{3/2}A^{2}\|_{L^{\infty}}
    \lesssim
    \|\sigma \rho ^{-1}|x|^{1/2}A\|_{L^{\infty}}^{2}
  \end{equation*}
  and the previous estimate simplifies to
  \begin{equation*}
    \|e^{itL}f\|_{L^{p}L^{q}}
    \lesssim
    (1+
    \|\sigma \rho^{-1}|x|^{1/2} A\|^{2}
    _{L^{\infty}\cap \dot H^{1/2}_{2n}}
    +
    \|\sigma\rho^{-1}|x|(V-i \partial \cdot A)\|_{L^{2n,\infty}}
    )
    \cdot
    \|f\|_{L^{2}}.
  \end{equation*}

  If we choose
  \begin{equation*}
    \sigma=\rho^{-1}|x|^{1/2},
    \qquad
    \rho=\bra{\log|x|}^{-\nu},
    \qquad
    \nu>1/2
  \end{equation*}
  we see that $\rho\in \ell^{2}L^{\infty}$,
  and
  \eqref{eq:assw}, \eqref{eq:assw2} (provided $\epsilon$
  is small enough), \eqref{eq:assw3}
  and \eqref{eq:assw5} are satisfied,
  as it follows by direct inspection using the
  basic properties of Muckenhoupt classes
  (see e.g. \cite{JohnsonNeugebauer91-a}). 
  Keeping into account the assumptions on the coefficients
  \eqref{eq:assVAendp}--\eqref{eq:assAendp},
  we have proved \eqref{eq:endpstrsch}. 

  The full range of
  indices $(p,q)$ is obtained immediately by interpolation 
  with the conservation of $L^{2}$ mass, and
  the nonhomogeneous estimate \eqref{eq:endostrschnh} 
  is proved by a standard
  application of the $TT^{*}$ method and the Christ--Kiselev
  Lemma, which is possible as long as $\widetilde{p}'<p$.
\end{proof}

By a gauge transformation we obtain the following
slightly more general result:

\begin{corollary}[]\label{cor:strschgauge}
  Let $n\ge3$ and $\phi:\mathbb{R}^{n}\to \mathbb{R}$
  such that $\partial \phi\in L^{n,\infty}$.
  Assume the operator $L$ defined in \eqref{eq:opL} is selfadjoint
  and non negative in $L^{2}(\mathbb{R}^{n})$, with domain
  $H^{2}(\mathbb{R}^{n})$.
  Assume $V:\mathbb{R}^{n}\to \mathbb{R}$
  and $A:\mathbb{R}^{n}\to \mathbb{R}^{n}$ satisfy
  for some $\delta>0$ and $\mu>1$
  \begin{equation*}
    \bra{\log|x|}^{\mu}\bra{x}^{\delta}
    |x|^{2}(V-i \partial \cdot A-i \Delta \phi)\in L^{\infty},
    \qquad
    \bra{\log|x|}^{\mu}\bra{x}^{\delta}|x| \widehat{B}\in L^{\infty}
  \end{equation*}
  and
  \begin{equation*}
    \bra{\log|x|}^{\mu}\bra{x}^{\delta}|x|(A+\partial \phi)
    \in L^{\infty}\cap \dot H^{1/2}_{2n}.
  \end{equation*}
  Moreover, assume 0 is not a resonance for $L$,
  in the sense of Definition \ref{def:reson}.

  Then the conclusion of Theorem \ref{the:endpointsch}
  are still valid for the Schr\"{o}dinger flow $e^{itL}$.
\end{corollary}

\begin{proof}
  Applying the gauge transformation
  \begin{equation*}
    u=e^{i \phi(x)}\widetilde{u}
  \end{equation*}
  and recalling \eqref{eq:gaugefor} we see that $\widetilde{u}$
  is a solution of
  \begin{equation*}
    \widetilde{u}(0)=e^{i \phi}f,
    \qquad
    i\partial_{t}\widetilde{u}-\Delta_{\widetilde{A}}
    \widetilde{u}
    -V\widetilde{u}=0,
    \qquad
    \widetilde{A}:=A+\partial \phi.
  \end{equation*}
  By assumption, $V$ and $\widetilde{A}$ satisfy the
  conditions of Theorem \ref{the:endpointsch} hence
  Strichartz estimates are valid for $\widetilde{u}$.
  Since Lebesgue and Lorentz norms of $u$ and $\widetilde{u}$
  coincide, the proof is concluded.
\end{proof}

\begin{theorem}[Strichartz for Wave]\label{the:striWE}
  Under the assumptions of Theorem \ref{the:endpointsch},
  or more generally
  the assumptions of Corollary \ref{cor:strschgauge},
  the wave flow $e^{it \sqrt{L}}$ satisfies the estimates
  \begin{equation}\label{eq:strWEL}
    \||D|^{\frac1q-\frac1p}e^{it \sqrt{L}}f\|_{L^{p}_{t}L^{q}}
    \lesssim
    \|f\|_{\dot H^{\frac12}},
    \qquad
    \||D|^{\frac1q-\frac1p}e^{it \sqrt{L}}L^{-1/2}f\|
      _{L^{p}_{t}L^{q}}
    \lesssim
    \|f\|_{\dot H^{-\frac12}}
  \end{equation}
  and
  \begin{equation}\label{eq:strWELnh}
    \textstyle
    \||D|^{\frac1q-\frac1p}\int_{0}^{t}e^{i(t-t')\sqrt{L}}F(t')dt'\|
    _{L^{p}_{t}L^{q}}
    \lesssim
    \||D|^{\frac1{\widetilde{p}}-\frac1{\widetilde{q}}}F\|
      _{L^{\widetilde{p}'}L^{\widetilde{q}'}}.
  \end{equation}
  for any wave admissible, non endpoint couples
  $(p,q)$ and $(\widetilde{p},\widetilde{q})$.
\end{theorem}

\begin{proof}
  When $\phi\neq0$,
  as in the previous proof, we perform a gauge transform
  $u=e^{i \phi}\widetilde{u}$;
  note that by Lemma \ref{lem:gaugeest} the transformation
  $u \mapsto \widetilde{u}$ is bounded on $\dot H^{s}_{p}$
  and on $\dot H^{-s}_{p'}$
  for $p<n/s$, $s\in[0,1]$ since $\partial \phi\in L^{n,\infty}$.
  Thus it is sufficient to prove the Strichartz estimates for
  $\widetilde{u}$. In the following we shall
  write $\widetilde{A}=A+\partial \phi$, but we shall omit
  for simplicity the tilde from $\widetilde{u}$
  and from $\widetilde{L}=-\Delta_{\widetilde{A}}+V$.
  Note that the wave flow satisfies the smoothing estimates
  \eqref{eq:smooWEh} and \eqref{eq:smooWE3}.

  The function $u=e^{it \sqrt{L}}$ solves
  \begin{equation*}
    \partial_{t}^{2}u=-Lu
    =\Delta u +2i \partial \cdot (\widetilde{A}u)
    +(V-|\widetilde{A}|^{2})u,
  \end{equation*}
  with data
  \begin{equation*}
    u(0)=f,\qquad
    \partial_{t}u(0)=i \sqrt{L}f.
  \end{equation*}
  Thus $u$ can be represented as
  \begin{equation}\label{eq:duhWE}
    \textstyle
    u=
    \cos(t|D|)f+
    i\frac{\sin(t|D|)}{|D|}\sqrt{L}f+
    \int_{0}^{t}\frac{\sin((t-t')|D|)}{|D|}
      (2i \partial \cdot(\widetilde{A}u)+
      (V-|\widetilde{A}|^{2})u)dt'.
  \end{equation}
  To the first term we apply directly the free estimate
  \begin{equation}\label{eq:strichWE}
    \||D|^{\frac1q-\frac1p}e^{-it|D|}f\|_{L^{p}_{t}L^{q}}
    \lesssim
    \|f\|_{\dot H^{\frac12}}
  \end{equation}
  To the second term we apply \eqref{eq:strichWE},
  obtaining
  \begin{equation}\label{eq:intermhom}
    \||D|^{-1}\sin(t|D|)\sqrt{L}f\|_{L^{p}_{t}L^{q}}
    \lesssim
    \||D|^{-\frac12}\sqrt{L}f\|_{L^{2}}.
  \end{equation}
  Since $|V|+|\widetilde{A}|^{2}\lesssim|x|^{-2}$, by Hardy's inequality
  we have
  \begin{equation*}
    \|L^{\frac12}f\|_{L^{2}}^{2}
    \le
    \|(\partial+iA)v\|_{L^{2}}^{2}
    +
    \||V|^{\frac12}f\|_{L^{2}}^{2}
    \lesssim
    \|f\|_{\dot H^{1}}^{2}
  \end{equation*}
  and by interpolation and duality we obtain
  \begin{equation}\label{eq:sobequiv}
    \|L^{\frac s2}f\|_{L^{2}}
    \lesssim
    \|f\|_{\dot H^{s}},
    \qquad
    \|f\|_{\dot H^{-s}}
    \lesssim
    \|L^{-\frac s2}f\|_{L^{2}},
    \qquad
    0\le s\le 1.
  \end{equation}
  Applying these inequalities to \eqref{eq:intermhom} we get
  \begin{equation*}
    \||D|^{-1}\sin(t|D|)\sqrt{L}f\|_{L^{p}_{t}L^{q}}
    \lesssim
    \|f\|_{\dot H^{\frac12}}.
  \end{equation*}

  Next we consider the last term in \eqref{eq:duhWE};
  more generally, we shall estimate two integrals of the form
  \begin{equation*}
    \textstyle
    \int_{0}^{t}\frac{e^{i(t-t')|D|}}{|D|}
      \partial\cdot(\widetilde{A}u)dt',
    \qquad
    \int_{0}^{t}\frac{e^{i(t-t')|D|}}{|D|}
      (V-|\widetilde{A}|^{2}) udt'.
  \end{equation*}
  Since we are 
  in the non endpoint case, by a standard application of the
  Christ--Kiselev Lemma, it will sufficient to estimate the
  two untruncated integrals
  \begin{equation*}
    \textstyle
    I=e^{it|D|}
    \int\frac{e^{-it'|D|}}{|D|}
      (V-|\widetilde{A}|^{2})udt',
    \qquad
    II=e^{it|D|}
    \int\frac{e^{-it'|D|}}{|D|}
      \partial \cdot(\widetilde{A} u)dt'.
  \end{equation*}

  If we first apply \eqref{eq:strichWE} then
  the dual smoothing estimate \eqref{eq:smooWEh} in the 
  elementary case $L=-\Delta$, we obtain
  \begin{equation}\label{eq:mixedWE}
    \textstyle
    \||D|^{\frac1q-\frac1p}e^{it|D|}
      \int \frac{e^{-it'|D|}}{|D|}F(t')dt'\|
      _{L^{p}_{t}L^{q}}
    \lesssim
    \|\rho^{-1}|x|^{1/2}|D|^{-1/2}F\|_{L^{2}_{t}L^{2}}.
  \end{equation}
  This gives
  \begin{equation*}
    \||D|^{\frac1q-\frac1p}II\| _{L^{p}_{t}L^{q}}
    \lesssim
    \|\rho^{-1}|x|^{1/2}|D|^{-1/2}
      \partial \cdot(\widetilde{A} u)
      \|_{L^{2}_{t}L^{2}}
    \lesssim
    \|\rho^{-1}|x|^{1/2}|D|^{1/2}
      (\widetilde{A} u) \|_{L^{2}_{t}L^{2}}
  \end{equation*}
  since $\rho^{-2}|x|\in A_{2}$. 
  We can commute
  $\rho^{-1}|x|^{1/2}$ with $|D|^{1/2}$ as in the proof of
  Theorem \ref{the:endpointsch}; we get
  \begin{equation*}
    \lesssim\||D|^{1/2}\rho^{-1}|x|^{1/2}
      (\widetilde{A} u) \|_{L^{2}_{t}L^{2}}
      =
    \||D|^{1/2}(\rho^{-2}|x|\widetilde{A})
    (\rho|x|^{-1/2}u)\|_{L^{2}_{t}L^{2}}
  \end{equation*}
  and by the Kato--Ponce inequality
  \begin{equation*}
    \lesssim
    \|\rho^{-2}|x|\widetilde{A}\|_{L^{\infty}}
    \||D|^{1/2}\rho|x|^{-1/2}u\|_{L^{2}_{t}L^{2}}
    +
    \||D|^{1/2}\rho^{-2}|x|\widetilde{A}\|_{L^{2n}}
    \|\rho|x|^{-1/2}u\|_{L^{2}_{t}L^{\frac{2n}{n-1}}}.
  \end{equation*}
  Applying \eqref{eq:muwleb} to the last term and recalling
  assumptions \eqref{eq:strWEL}, \eqref{eq:strWELnh}
  we obtain
  \begin{equation*}
    \lesssim
    \||D|^{1/2}\rho|x|^{-1/2}u\|_{L^{2}_{t}L^{2}}
    +
    \|\rho|x|^{-1/2}|D|^{1/2}u\|_{L^{2}_{t}L^{2}}
    \lesssim
    \|L^{1/4}f\|_{L^{2}}
    \lesssim
    \|f\|_{\dot H^{1/2}}
  \end{equation*}
  by the smoothing estimates \eqref{eq:smooWEh}
  and \eqref{eq:smooWE3}.

  For the remaining term $I$ we have, again by
  \eqref{eq:mixedWE}.
  \begin{equation*}
    \||D|^{\frac1q-\frac1p}I\| _{L^{p}_{t}L^{q}}
    \lesssim
    \|\rho^{-1}|x|^{1/2}|D|^{-1/2}
    (V-|\widetilde{A}|^{2})u\|_{L^{2}_{t}L^{2}}.
  \end{equation*}
  Then we repeat exactly the same steps as in the estimate
  of the term $I$ in the proof of Theorem \ref{the:endpointsch}
  (with $\sigma=\rho^{-1}|x|^{1/2}$),
  and we arrive at
  \begin{equation*}
    \lesssim
    \|\rho|x|^{-1/2}|D|^{1/2}u\|_{L^{2}_{t}L^{2}}
    \lesssim\|L^{1/4}f\|_{L^{2}}\lesssim\|f\|_{L^{2}}.
  \end{equation*}
  using \eqref{eq:smooWEh}.
  Summing up, we have proved the first estimate in
  \eqref{eq:strWEL}.

  The proof of the second estimate in \eqref{eq:strWEL}
  is completely identical: indeed,
  it is sufficient to notice that
  $u=\sin(t \sqrt{L})L^{-1/2}$ solves
  \begin{equation*}
    \partial_{t}^{2}u=-Lu,
    \qquad
    u(0)=0,
    \qquad
    \partial_{t}u(0)=f.
  \end{equation*}
  The proof of the nonhomogeneous estimate \eqref{eq:strWELnh}
  follows as usual by a $TT^{*}$ argument and the
  Christ--Kiselev Lemma (since $\widetilde{p}'<2<p$).
\end{proof}

\begin{theorem}[Strichartz for K--G]\label{the:striKG}
  Under the assumptions of Theorem \ref{the:endpointsch},
  or of Corollary \ref{cor:strschgauge},
  the Klein--Gordon flow $e^{it \sqrt{L+1}}$ 
  satisfies the estimates
  \begin{equation}\label{eq:strKGL}
    \|\bra{D}^{\frac1q-\frac1p}e^{it \sqrt{L+1}}f\|_{L^{p}_{t}L^{q}}
    \lesssim
    \|f\|_{ H^{\frac12}},
    \qquad
    \|\bra{D}^{\frac1q-\frac1p}e^{it \sqrt{L+1}}L^{-1/2}f\|
      _{L^{p}_{t}L^{q}}
    \lesssim
    \|f\|_{ H^{-\frac12}}
  \end{equation}
  and
  \begin{equation}\label{eq:strKGLnh}
    \textstyle
    \|\bra{D}^{\frac1q-\frac1p}\int_{0}^{t}e^{i(t-t')\sqrt{L+1}}F(t')dt'\|
    _{L^{p}_{t}L^{q}}
    \lesssim
    \|\bra{D}^{\frac1{\widetilde{p}}-\frac1{\widetilde{q}}}F\|
      _{L^{\widetilde{p}'}L^{\widetilde{q}'}}.
  \end{equation}
  for any wave or Schr\"{o}dinger admissible, non endpoint couples
  $(p,q)$ and $(\widetilde{p},\widetilde{q})$.
\end{theorem}

\begin{proof}
  The proof is almost identical to the proof of 
  Theorem \ref{the:striWE},
  and is based on the estimate
  \begin{equation*}
    \|\bra{D}^{\frac1q-\frac 1p}
        e^{it \sqrt{1-\Delta}}f\|_{L^{p}_{t}L^{q}_{x}}
    \lesssim
    \|f\|_{ H^{\frac12}(\mathbb{R}^{m})}
  \end{equation*}
  which holds both if the couple $(p,q)$ is
  wave admissible and
  if it is Schr\"{o}dinger admissible.
  A complete proof for Schr\"{o}dinger admissible $(p,q)$
  can be found e.g. in the Appendix of \cite{DAnconaFanelli08-a},
  while for wave admissible indices the proof is obtained
  starting from the estimate
  \begin{equation*}
    j\ge1,\quad
    \phi_{j}\in \mathscr{S},
    \quad
    \spt \widehat{\phi_{j}}=\{|\xi|\sim 2^{j}\}
    \quad\implies\quad
    \|e^{it \sqrt{1-\Delta}}\phi_{j}\|
        _{L^{\infty}(\mathbb{R}^{m})}
    \lesssim
    |t|^{-\frac{m-1}{2}}2^{\frac{m+1}{2}}
  \end{equation*}
  (see \cite{Brenner85-a}) and then applying the
  usual Ginibre-Velo procedure.
\end{proof}

\bibliography{/Users/piero/Documents/Biblioteca/-bib/bibliodatabase.bib}
\bibliographystyle{abbrv}
\end{document}